\documentclass[reqno]{amsart}

\usepackage{graphicx}
\usepackage{amscd}
\usepackage{amsmath, amssymb}
\usepackage{graphics}
\usepackage{graphicx}
\usepackage{epsfig}
\usepackage{xcolor}
\usepackage{soul}
\usepackage{tikz}

\newtheorem{theorem}{Theorem}

\newtheorem{definition}[theorem]{Definition}

\newtheorem{lemma}[theorem]{Lemma}

\newtheorem{proposition}[theorem]{Proposition}
\newtheorem{remark}[theorem]{Remark}

\newtheorem{question}[theorem]{Question}

\begin{document}

\title[The virtual singular twin monoid and group]{The virtual singular twin monoid and group: presentations and representations}

\author{Carmen Caprau}
\address{Carmen Caprau\\
               Department of Mathematics\\
                California State University, Fresno\\
                5245 N. Backer Ave. M/S PB 108\\
                Fresno, ca 93740, USA}

\email{ccaprau@csufresno.edu}

\author{Mohamad N. Nasser}
\address{Mohamad N. Nasser\\
         Department of Mathematics and Computer Science\\
         Beirut Arab University\\
         P.O. Box 11-5020, Beirut, Lebanon}
         
\email{m.nasser@bau.edu.lb}

\maketitle
 
\begin{abstract}
In this article, we introduce the algebraic definitions and presentations of the virtual singular twin monoid and virtual singular twin group, denoted by $VSTM_n$ and $VST_n$, respectively, for a positive integer $n$. These structures extend the twin group $T_n$ in close analogy  to how the virtual singular braid monoid and virtual singular braid group extend the classical braid group. We then construct and study representations of the group $VST_n$, for $n \geq 3$, focusing in particular on extending the representations $\eta_1$ and $\eta_2$ of $T_n$, introduced by M. Nasser, to $VST_n$ via the $2$-local extension method. To analyze the resulting representations, $\eta_1'$ and $\eta_2'$, and their properties, we establish necessary and sufficient conditions for irreducibility and show that both $\eta_1'$ and $\eta_2'$ are unfaithful. Additionally, we classify all complex homogeneous $2$-local representations of $VST_n$ for every integer $n\geq 3$, providing a foundation for further investigation into representations of $VST_n$.
\end{abstract}

\renewcommand{\thefootnote}{}

\section{Introduction}

The braid group on $n$ strands, denoted $B_n$, was introduced in 1926 by E. Artin and is generated by the elements $\sigma_1,\sigma_2, \ldots,\sigma_{n-1}$, known as the classical braid generators (see  \cite{Artin1, Artin2}). In the 1990s, J. Baez in \cite{baezz} and J. Birman in \cite{birman}, independently, introduced the singular braid monoid $SM_n$, which is the monoid generated by the classical braid elements $\sigma_1^{\pm 1},\sigma_2^{\pm 1}, \ldots,\sigma_{n-1}^{\pm 1}$ of $B_n$ along with a set of singular generators $\tau_1,\tau_2, \ldots, \tau_{n-1}$. In 1998, R. Fenn, E. Keyman, and C. Rourke proved that $SM_n$ embeds into a group, denoted $SB_n$, called the singular braid group \cite{Fenn}.

Building on this framework, in 2016, C. Caprau et al. introduced in~\cite{Caprau} the virtual singular braid monoid $VSM_n$, which is generated by the classical braid generators $\sigma_1^{\pm 1}, \sigma_2^{\pm 1}, \ldots, \sigma_{n-1}^{\pm 1}$, the singular generators $\tau_1, \tau_2, \ldots, \tau_{n-1}$, and an additional set of virtual generators $\nu_1^{\pm 1}, \nu_2^{\pm 1}, \ldots, \nu_{n-1}^{\pm 1}$. In 2024, C. Caprau and A. Yeung showed that $VSM_n$ embeds into a group, denoted by $VSB_n$, called the virtual singular braid group \cite{Caprau2}. The  monoid $VSM_n$ has as subsets the virtual braid group $VB_n$~\cite{Ka, Kauffman, KL1, KL2} and the singular braid monoid $SM_n$. 

The algebraic structures $B_n$, $VB_n$, $SM_n$, $SB_n$, $VSM_n$, and $VSB_n$ have found broad applications in low-dimensional topology and geometric studies. The braid group $B_n$ plays a central role in knot theory, as every link can be represented as the closure of a braid \cite{Alexander, Markov}. It also arises naturally in the study of configuration spaces and mapping class groups \cite{BirmanBook}. The singular braid monoid $SM_n$ offers a powerful algebraic framework for studying singular knots and Vassiliev invariants \cite{baezz,birman}. On the other hand, the virtual braid group $VB_n$ and virtual singular braid monoid $VSM_n$ extend these ideas into the realm of virtual knot theory, enabling the study of curve embeddings and curve immersions, respectively, in thickened surfaces and their isotopy classes \cite{Kauffman}. These developments illustrate how generalizations of the classical braid group enrich the interplay between algebraic structures and geometric topology.

Another important group to consider is the twin group on $n$ strands, denoted $T_n$, which is a right-angled Coxeter group generated by the elements $s_1, s_2, \ldots, s_{n-1}$. This group was first introduced in 1990 by G. Shabat and V. Voevodsky \cite{Shabat1990}, and later studied by M. Khovanov \cite{Khovanov1996, Khovanov1997}, who provided a geometric interpretation analogous to that of the braid group $B_n$. The same algebraic structure has since appeared under different names in the literature, such as the flat braid group in \cite{Merkov1999} and the planar braid group in \cite{Mosto2020}. More recently,  in 2025, M. Nasser and N. Chbili \cite{NasserNafaa} introduced the singular twin group $ST_n$, as an extension of the twin group obtained by adjoining the family of singular generators $\tau_1, \tau_2, \ldots, \tau_{n-1}$, which satisfy relations analogous to those defining the singular braid group.

The twin group and its extensions play an important role in the study of geometric and topological structures such as doodles and flat knots. In particular, the group $T_n$ captures the combinatorial and isotopic properties of configurations of immersed curves in the plane that lack over/under crossing information~\cite{Khovanov1997, Fenn2016}. These structures provide  algebraic models for analyzing planar link diagrams, Reidemeister-type moves, and equivalence classes of doodles on surfaces. The singular twin monoid further enriches this framework by introducing singular interactions between strands, offering new perspectives in the study of singular doodles and related knot-theoretic structures.
 
The study of these algebraic structures naturally leads to an investigation of their representations from an algebraic perspective. Representation theory provides a powerful framework for understanding groups and monoids, such as the braid groups and twin groups, by associating each generator with a matrix in such a way that the algebraic relations are preserved. Key characteristics of such representations include faithfulness, which ensures that distinct algebraic elements correspond to distinct matrices (i.e., the representation is injective), and irreducibility, which guarantees that the representation does not admit a proper nontrivial invariant subspace (i.e., the representation does not have a nontrivial subrepresentation).

The goal of this paper is to introduce new algebraic structures that extend both the twin group and the singular twin group. In Section~\ref{sec:background}, we provide a literature review for the braid groups and  twin groups, and introduce the notion of $k$-local representations, with particular focus on known homogeneous $2$-local representations. In Section~\ref{sec:new_monoid_and_group}, we introduce the virtual singular twin monoid $VSTM_n$ and the virtual singular twin group $VST_n$ via generators and relations. In Section~\ref{sec:rep}, we explore representations of the virtual singular twin group, focusing on extensions of the known representations $\eta_1$ and $\eta_2$ of the twin group $T_n$ into $VST_n$ via $2$-local extension methods (Theorems~\ref{Theta1} and~\ref{Theta2}). We then establish necessary and sufficient conditions for these new representations to be irreducible (Theorems~\ref{eta1irr} and~\ref{eta2irr}). In Section~\ref{sec:complex_rep}, we classify all complex homogeneous $2$-local representations of $VST_n$ for all $n\geq 3$, laying the groundwork for future investigations into the broader representation theory of $VST_n$. Finally, in Section~\ref{sec:different_presentations}, we provide additional presentations for the virtual singular twin monoid $VSTM_n$.

\section{Braid Groups, Twin Groups, and Local Representations} \label{sec:background}

We start by providing a brief overview of the relevant monoids and groups—along with some of their representations—that play a role in the development of our results.

\begin{definition}
The singular braid monoid on $n$ strands~\cite{baezz, birman}, denoted $SM_n$, is the monoid generated by the elements $\{\sigma_i, \sigma_i^{-1}, \tau_i \, | \, 1 \leq i \leq n-1\}$, and subject to the following relations:
 \begin{align}
 \sigma_i\sigma_j\sigma_i &= \sigma_j\sigma_i\sigma_j ,\hspace{1.1cm} |i-j| = 1,  \label{eqs1}\\
\sigma_i\sigma_j &= \sigma_j\sigma_i , \hspace{1.45cm} |i-j|\geq 2, \label{eqs2}\\
 \sigma_i\sigma_i^{-1} &=\sigma_i^{-1}\sigma_i ,\hspace{1.1cm}  \text{ for all } i, \label{eqs3}\\
\tau_i\tau_j &=\tau_j\tau_i ,\hspace{1.55cm} |i-j|\geq 2, \label{eqs4}\\ 
\tau_i\sigma_j &=\sigma_j\tau_i ,\hspace{1.5cm} |i-j|\geq 2, \label{eqs5}\\
 \tau_i\sigma_i &=\sigma_i\tau_i ,\hspace{1.4cm} \text{ for all } i, \label{eqs6}\\
 \sigma_i\sigma_j\tau_i &=\tau_j\sigma_i\sigma_j, \hspace{1.1cm} |i-j| = 1.  \label{eqs7}
\end{align}
The braid group on $n$ strands \cite{Artin1, Artin2}, denoted $B_n$, is the group with presentation 
\[ <\sigma_1, \ldots,  \sigma_{n-1} \, | \, \text{Relations} \, \eqref{eqs1}, \eqref{eqs2}>. \]
\noindent By imposing that the generators $\tau_i$ are invertible for all $1\leq i \leq n-1,$ we obtain a group that extends $B_n$. This group is generated by $\{\sigma_i, \tau_i \, |  1 \leq i \leq n-1\}$, and is known as the singular braid group, denoted $SB_n$ \cite{Fenn}. The elements $\sigma_i$ and $\tau_i$ are called the braid generators and, respectively, the singular generators. 
\end{definition}

Figure~\ref{SingBraidGenerators} shows diagrammatic representations of the braid generators $\sigma_i$ and singular generators $\tau_i$, along with their inverses $\sigma_i^{-1}$ and $\tau_i^{-1}$, respectively. We remark that the convention for the diagrammatic representations for $\sigma_i$ and $\sigma_i^{-1}$ is reversed in other works, but that does not affect the results we work with here.

\begin{figure}[h!]
\begin{tikzpicture}
	\draw[thick] (-1.5,0)--(-1.5,2); 
    \fill (-1,1) circle(1pt) (-1.2,1) circle(1pt)(-0.8,1)circle(1pt);       
    \draw[thick] (-.5,0)--(-.5,2);
    \draw[thick] (2.5,0)--(2.5,2); 
    \fill (2,1) circle(1pt) (2.2,1) circle(1pt) (1.8,1)circle(1pt);       
    \draw[thick] (1.5,0)--(1.5,2);
    \draw[thick] (0,0) to[out=90, in=-90] (1,2);
    \draw[thick, white, line width=4pt] (1,0) to[out=90, in=-90] (0,2); 
    \draw[thick] (1,0) to[out=90, in=-90] (0,2);
    \node[above] at(0,2){$i$};
    \node[above] at(1,2){$i+1$};
     \node[left] at(-2,1){$\sigma_i=$};
	\node at (2.8,0.8){,};  
 	\draw[thick] (5.5,0)--(5.5,2); 
    \fill (5,1) circle(1pt) (4.8,1) circle(1pt)(5.2,1)circle(1pt);       
    \draw[thick] (4.5,0)--(4.5,2);
    \draw[thick] (8.5,0)--(8.5,2); 
    \fill (8,1) circle(1pt) (8.2,1) circle(1pt)(7.8,1)circle(1pt);       
    \draw[thick] (7.5,0)--(7.5,2);    
    \draw[thick] (7,0) to[out=90, in=-90] (6,2);
    \draw[thick, white, line width=4pt] (6,0) to[out=90, in=-90] (7,2); 
    \draw[thick] (6,0) to[out=90, in=-90] (7,2);
    \node[above] at(6,2){$i$};
    \node[above] at (7,2) {$i+1$};
     \node[left] at(4.5,1){$\sigma_i^{-1}=$};
\end{tikzpicture}

\begin{tikzpicture}
	\draw[thick] (5.5,0)--(5.5,2); 
   \fill (5,1) circle(1pt) (4.8,1) circle(1pt)(5.2,1)circle(1pt);       
    \draw[thick] (4.5,0)--(4.5,2);
    \draw[thick] (8.5,0)--(8.5,2); 
    \fill (8,1) circle(1pt) (8.2,1) circle(1pt)(7.8,1)circle(1pt);       
   \draw[thick] (7.5,0)--(7.5,2);    
	\draw[thick] (7,0) to[out=90, in=-90] (6,2);
    \draw[thick, white, line width=4pt] (6,0) to[out=90, in=-90] (7,2); 
   \draw[thick] (6,0) to[out=90, in=-90] (7,2);
\draw[draw=black,fill=white] (6.5,1) circle (2.5pt); 
    \node[above] at(6,2){$i$};
    \node[above] at (7,2) {$i+1$};
     \node[left] at(-2,1){$\tau_i=$};
	\node at (2.8,0.8){,};    
 	\draw[thick] (-1.5,0)--(-1.5,2); 
    \fill (-1,1) circle(1pt) (-1.2,1) circle(1pt)(-0.8,1)circle(1pt);       
    \draw[thick] (-.5,0)--(-.5,2);
    \draw[thick] (2.5,0)--(2.5,2); 
    \fill (2,1) circle(1pt) (2.2,1) circle(1pt)(1.8,1)circle(1pt);       
    \draw[thick] (1.5,0)--(1.5,2);
    \draw[thick] (0,0) to[out=90, in=-90] (1,2);
    \draw[thick, white, line width=4pt] (1,0) to[out=90, in=-90] (0,2); 
    \draw[thick] (1,0) to[out=90, in=-90] (0,2);
    \fill[black] (0.5,1) circle (2.5pt); 
    \node[above] at(0,2){$i$};
    \node[above] at(1,2){$i+1$};
     \node[left] at(4.5,1){$\tau_i^{-1}=$};
\end{tikzpicture}
\caption{The generators \(\sigma_i\) and \(\tau_i\) and their inverses}
\label{SingBraidGenerators}
\end{figure}
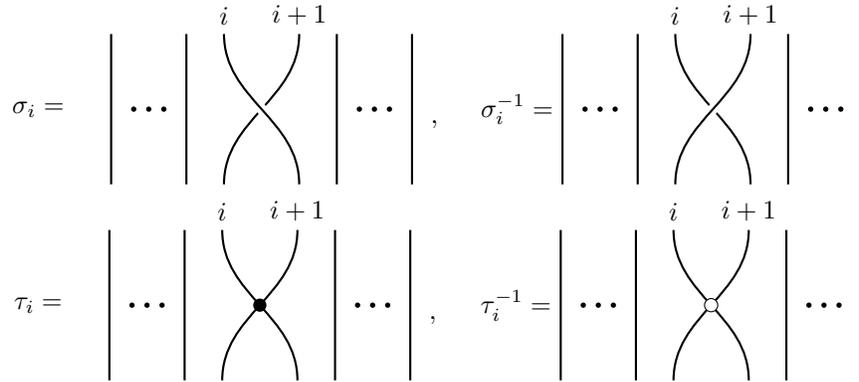

When working diagrammatically with elements of $SB_n$ (or $B_n$), we refer to them as singular braids (or classical braids) on $n$ strands. We compose these braids by stacking one on top of the other and connecting their endpoints.

\begin{definition}
The virtual singular braid monoid on $n$ strands~\cite{Caprau, Caprau1, Caprau2}, denoted here by $VSM_n$, is the monoid generated by the elements $\sigma_i, \sigma_i^{-1}$ and $\tau_i$ of $SM_n$, together with the family of elements $\nu_i$, for all $1\leq i \leq n-1$. In addition to the relations \eqref{eqs1} through \eqref{eqs7} of  the monoid $SM_n$, the generators of $VSM_n$ are subject to the following relations:
 \begin{align}
\nu_i^2 &=1 ,\hspace{1.75cm} \text{ for all } i,  \label{eqs8}\\
\nu_i\nu_j\nu_i &= \nu_j\nu_i\nu_j ,\hspace{1.1cm} |i-j| = 1, \label{eqs9}\\
 \nu_i\sigma_j\nu_i &= \nu_j\sigma_i\nu_j ,\hspace{1.1cm} |i-j| = 1, \label{eqs10}\\
 \nu_i\tau_j\nu_i &= \nu_j\tau_i\nu_j ,\hspace{1.15cm} |i-j| = 1,  \label{eqs11}\\
\nu_i\nu_j &=\nu_j\nu_i ,\hspace{1.45cm} |i-j|\geq 2,  \label{eqs12}\\
 \nu_i\sigma_j &=\sigma_j\nu_i ,\hspace{1.4cm} |i-j|\geq 2, \label{eqs13}\\
\nu_i\tau_j &=\tau_j\nu_i ,\hspace{1.5cm} |i-j|\geq 2. \label{eqs14} 
\end{align} 
 By imposing that the generators $\tau_i$ are invertible and adjoining their inverses, we obtain a group extension of both $B_n$ and $SB_n$, generated by $\{\sigma_i, \tau_i, \nu_i \, | \, 1\leq i \leq n-1\}$. This group is called the virtual singular braid group and is denoted by $VSB_n$ \cite{Caprau2}.
 
 The virtual braid group on $n$ strands~\cite{Ka, Kauffman, KL1, KL2}, denoted by $VB_n$, is the group generated by $\{\sigma_i, \nu_i \, | \, 1\leq i\leq n-1 \}$ and subject to the relations ~\eqref{eqs1}, \eqref{eqs2}, \eqref{eqs8}, \eqref{eqs9}, \eqref{eqs10}, \eqref{eqs12} and \eqref{eqs13}.
 \end{definition}

 The virtual generators $\nu_i$ are usually depicted diagrammatically as in Fig.~\ref{VirtualGenerators}.
 \begin{figure}[ht]
\begin{tikzpicture}
\draw[thick] (5.5,0)--(5.5,2); 
\fill (5,1) circle(1pt) (4.8,1) circle(1pt)(5.2,1)circle(1pt);       
\draw[thick] (4.5,0)--(4.5,2);
\draw[thick] (8.5,0)--(8.5,2); 
\fill (8,1) circle(1pt) (8.2,1) circle(1pt)(7.8,1)circle(1pt);       
\draw[thick] (7.5,0)--(7.5,2);    
\draw[thick] (7,0) to[out=90, in=-90] (6,2);
\draw[thick] (6,0) to[out=90, in=-90] (7,2);
\draw[thick] (6.5,1) circle (3pt);
\node[above] at(6,2){$i$};
\node[above] at (7,2) {$i+1$};
 \node[left] at(4,1){$\nu_i=$};
\end{tikzpicture}
\caption{The generator $\nu_i$}
\label{VirtualGenerators}
\end{figure}
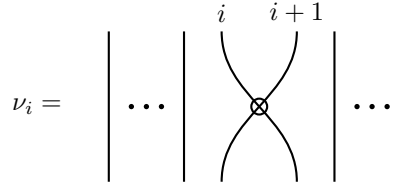

\begin{definition}\cite{Shabat1990}
The twin group on $n$ strands, denoted $T_n$, is the group generated by the elements $s_i$, where $1\leq i\leq n-1$, and subject to the following relations:
\begin{align}
 s_i^2 &=1,\hspace{0.85cm} \text{ for all } i,  \label{eqs15}\\
s_is_j &=s_js_i, \hspace{0.55cm} |i-j|\geq 2. \label{eqs16}
\end{align}
\end{definition}

We represent the generators $s_i$ diagrammatically as in Fig.~\ref{TwinGenerators}, and we refer to them as \textit{twin elements}.
\begin{figure}[h!]
\begin{tikzpicture}
\draw[thick] (5.5,0)--(5.5,2); 
\fill (5,1) circle(1pt) (4.8,1) circle(1pt)(5.2,1)circle(1pt);       
\draw[thick] (4.5,0)--(4.5,2);
\draw[thick] (8.5,0)--(8.5,2); 
\fill (8,1) circle(1pt) (8.2,1) circle(1pt)(7.8,1)circle(1pt);       
\draw[thick] (7.5,0)--(7.5,2);    
\draw[thick] (7,0) to[out=90, in=-90] (6,2);
\draw[thick] (6,0) to[out=90, in=-90] (7,2);
\node[above] at(6,2){$i$};
\node[above] at (7,2) {$i+1$};
 \node[left] at(4,1){$s_i=$};
\end{tikzpicture}
\caption{The generator $s_i$}
\label{TwinGenerators}
\end{figure}
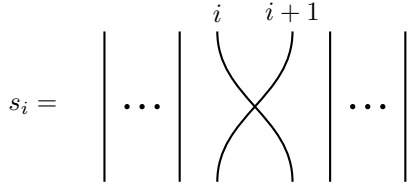

\begin{definition} \cite{NasserNafaa}
The singular twin monoid on $n$ strands, denoted $STM_n$, is a monoid generated by $\{s_i, \tau_i \, | \, 1\leq i \leq n-1 \}$, where $s_i$ are the twin elements of $T_n$ and $\tau_i$ are the singular elements of $SM_n$. In addition to the relations \eqref{eqs15} and \eqref{eqs16} of $T_n$,  and relation \eqref{eqs4} of $SM_n$, the generators of $STM_n$ satisfy the following relations:
\begin{align}
\tau_is_j &=s_j\tau_i ,\hspace{1.5cm} |i-j| \geq 2,  \label{eqs17}\\
 \tau_is_i &=s_i\tau_i ,\hspace{1.4cm} \text{ for  all } i,  \label{eqs18}\\
s_is_j\tau_i &=\tau_js_is_j, \hspace{1.2cm} |i-j| = 1. \label{eqs19}
\end{align}
 By imposing that the generators $\tau_i$ are invertible and adjoining their inverses, we obtain a group extension of $T_n$, generated by  $\{s_i, \tau_i\, | \, 1\leq i \leq n-1 \}$. This group is called the singular twin group and is denoted by $ST_n$.
\end{definition}

We now proceed to introduce the concept of $k$-local representations of a group $G$, which will serve as a key tool in this paper.

\vspace*{0.1cm}

\begin{definition} \cite{40} \label{deflocal}
Let $t$ be an indeterminate and let $G$ be a group defined by a presentation with generators $g_1,g_2,\ldots,g_{n-1}.$ A representation $\theta: G \rightarrow \mathrm{GL}_{m}(\mathbb{Z}[t^{\pm 1}])$ is said to be $k$-local if 
$$\theta(g_i) =\left( \begin{array}{c|@{}c|c@{}}
   \begin{matrix}
     I_{i-1} 
   \end{matrix} 
      & 0 & 0 \\
      \hline
    0 &\hspace{0.2cm} \begin{matrix}
   		M_i
   		\end{matrix}  & 0  \\
\hline
0 & 0 & I_{n-i-1}
\end{array} \right) \hspace*{0.2cm} \text{for all} \hspace*{0.2cm} 1\leq i\leq n-1,$$ 
where $M_i \in \mathrm{GL}_k(\mathbb{Z}[t^{\pm 1}])$ with $k=m-n+2$ and $I_r$ is the $r\times r$ identity matrix. The $k$-local representation is said to be homogeneous if all the matrices $M_i$ are equal.
\end{definition}

\begin{remark}
The concept of $k$-local representations can be extended to a group $G$ generated by $q(n-1)$ elements, where these elements are partitioned into $q$ disjoint classes, each consisting of $n-1$ generators. For example, if $q=3$, $t$ is an indeterminate and $G$ is a group with generators $\{g_i, h_i, p_i \, | \, 1\leq i\leq n-1 \}$, a $k$-local representation $\theta: G \rightarrow \mathrm{GL}_{m}(\mathbb{Z}[t^{\pm 1}])$ is a representation given by
$$\theta(g_i) =\left( \begin{array}{c|@{}c|c@{}}
   \begin{matrix}
     I_{i-1} 
   \end{matrix} 
      & 0 & 0 \\
      \hline
    0 &\hspace{0.2cm} \begin{matrix}
   		G_i
   		\end{matrix}  & 0  \\
\hline
0 & 0 & I_{n-i-1}
\end{array} \right), 
\hspace{1cm}
\theta(h_i) =\left( \begin{array}{c|@{}c|c@{}}
   \begin{matrix}
     I_{i-1} 
   \end{matrix} 
      & 0 & 0 \\
      \hline
    0 &\hspace{0.2cm} \begin{matrix}
   		H_i
   		\end{matrix}  & 0  \\
\hline
0 & 0 & I_{n-i-1}
\end{array} \right),$$ 
and
$$\theta(p_i) =\left( \begin{array}{c|@{}c|c@{}}
   \begin{matrix}
     I_{i-1} 
   \end{matrix} 
      & 0 & 0 \\
      \hline
    0 &\hspace{0.2cm} \begin{matrix}
   		P_i
   		\end{matrix}  & 0  \\
\hline
0 & 0 & I_{n-i-1}
\end{array} \right),$$ 
for $1\leq i\leq n-1,$ where $G_i,H_i,P_i \in \mathrm{GL}_k(\mathbb{Z}[t^{\pm 1}])$, $k=m-n+2$, and $I_r$ is the $r\times r$ identity matrix. In addition, $\theta$ is homogeneous if  $G_i = G_j, H_i = H_j$ and $P_i = P_j$, for all $i \neq j$.    
\end{remark}

\begin{question}
Given a homogeneous $k$-local representation $\theta$ of a group $G$, what structural features and algebraic properties can $\theta$ exhibit?
\end{question}

Recently, the study of $k$-local representations has seen significant development. The investigation of such representations for braid groups began with the work of Y. Mikhalchishina in~\cite{Mik}, who classified all complex $2$-local representations of $B_3$, as well as all complex homogeneous $2$-local representations of $B_n$, for all $n \geq 3$. Building on this, T. Mayassi and M. Nasser investigated in \cite{37} the complex homogeneous $3$-local representations of $B_n$ for all $n \geq 4$. In the same work, they also classified all complex homogeneous $2$-local representations of the singular braid monoid $SM_n$ for all $n \geq 2$, and all complex homogeneous $3$-local representations of $SM_n$ for all $n \geq 4$. 

On the other hand, regarding $k$-local representations of the twin groups, T. Mayassi and M. Nasser provided a full classification of the $2$-local representations of $T_n$ for all $n \geq 2$ in \cite{Mayasi20251}. Further progress was made by M. Nasser, who classified in \cite{M.Nass.3twin} all $3$-local representations of the twin group $T_n$, the virtual twin group $VT_n$, and the welded twin group $WT_n$, for all $n \geq 4$.  Additionally, in \cite{M.N.twin}, M. Nasser introduced two specific representations of $T_n$, denoted by $N_1$ and $N_2$, for all $n \geq 2$. 

We provide now the definitions and key properties of these representations.
 
\begin{definition} \cite{M.N.twin} \label{defeta1}
The $N_1$-representation $\eta_1: T_n \rightarrow \mathrm{GL}_n(\mathbb{Z}[t^{\pm 1}])$, where $t$ is an indeterminate, is the representation defined by
$$\eta_1(s_i) = \left( \begin{array}{c|@{}c|c@{}}
   \begin{matrix}
     I_{i-1} 
   \end{matrix} 
      & 0 & 0 \\
      \hline
    0 &\hspace{0.2cm} \begin{matrix}
   	1-t & t\\
   	2-t & t-1\\
\end{matrix}  & 0  \\
\hline
0 & 0 & I_{n-i-1}
\end{array} \right) \text{ for  } 1\leq i \leq n-1.$$ 
\end{definition}

\begin{theorem} \cite{M.N.twin}
The representation $\eta_1: T_n \rightarrow \mathrm{GL}_n(\mathbb{Z}[t^{\pm 1}])$ is reducible to the degree $n-1$ for all $n\geq 3$, and it is faithful for $n=2$ and $n=3$. Moreover, the complex specialization of its $(n-1)$-composition factor, namely $\eta_1': T_n \rightarrow \mathrm{GL}_{n-1}(\mathbb{C})$, is irreducible if and only if $t\neq \dfrac{2n-2}{n-2}$ and $t\neq 2$.
\end{theorem}

 
\begin{definition}\cite{M.N.twin} \label{defeta2}
The $N_2$-representation $\eta_2: T_n \rightarrow \mathrm{GL}_n(\mathbb{Z}[t^{\pm 1}])$, where $t$ is an indeterminate, is the representation defined by
$$\eta_2(s_i) = \left( \begin{array}{c|@{}c|c@{}}
   \begin{matrix}
     I_{i-1} 
   \end{matrix} 
      & 0 & 0 \\
      \hline
    0 &\hspace{0.2cm} \begin{matrix}
   	0 & f(t)\\
   	\dfrac{1}{f(t)} & 0\\
\end{matrix}  & 0  \\
\hline
0 & 0 & I_{n-i-1}
\end{array} \right) \text{ for } 1\leq i \leq n-1,$$
where $f(t)\in \mathbb{Z}[t^{\pm 1}]$ with $f(t)$ is invertible in $\mathbb{Z}[t^{\pm 1}]$.
\end{definition}

\begin{theorem} \cite{M.N.twin} \label{eta2fai}
The representation $\eta_2: T_n \rightarrow \text{GL}_n(\mathbb{Z}[t^{\pm 1}])$ is reducible. Moreover, $\eta_2$ is faithful for $n=2$ and unfaithful for all $n\geq 3$.
\end{theorem}

\section{The Virtual Singular Twin Monoid and Group} \label{sec:new_monoid_and_group}

In analogy with the constructions of the virtual singular braid monoid and the virtual singular braid group, we introduce in this section two related algebraic structures: the virtual singular twin monoid and the virtual singular twin group. These structures are defined in a manner that mirrors their braid counterparts, while exhibiting distinct properties and relations arising from the twin group framework. Our objective is to formally define these structures and to investigate their generating sets and defining relations.

\begin{definition} \label{def:VSTM}
The virtual singular twin monoid on $n$ strands, denoted $VSTM_n$, is the monoid generated by the elements $\{s_i, \tau_i, \nu_i \, | \, 1 \leq i\leq n-1 \}$, and satifying the following relations:
\begin{align}
s_i^2 &=1 \text{ and } \nu_i^2 =1, \text{ for all } i \label{eqs20}\\
\tau_is_i &=s_i\tau_i,  \hspace{0.4cm} \text{ for all } i \label{eqs21}\\
s_is_j \tau_i &= \tau_js_is_j, \text{ for } |i-j| = 1 \label{eqs22} \\
\nu_i\nu_j\nu_i &= \nu_j\nu_i\nu_j, \text{ for } |i-j| = 1 \label{eqs23} \\
\nu_is_j\nu_i &= \nu_js_i\nu_j,  \text{ for } |i-j| = 1 \label{eqs24} \\
\nu_i\tau_j\nu_i &= \nu_j\tau_i\nu_j, \text{ for } |i-j| = 1 \label{eqs25} \\
g_ih_j &= h_jg_i, \,\,\,\, \text{ for } |i-j| \geq 2, \text{ where } g_i,h_i \in \{s_i, \tau_i, \nu_i \}.
\end{align}
By imposing that the generators $\tau_i$ are invertible and adjoining their inverses $\tau_i^{-1}$, we obtain a group extension of both the twin group $T_n$ and the singular twin group $ST_n$. We refer to this group as the virtual singular twin group, and we denote it by $VST_n$.
\end{definition}

When regarding the elements of $VSTM_n$ and $VST_n$ diagrammatically, we refer to them as \textit{virtual singular twin braids}. Such braids are composed analogously to those in the previously defined monoids and groups, by placing one braid on top of another and connecting their endpoints.

We remark that since $s_i^2 = 1$ for all $i$, the relation $s_is_j \tau_i = \tau_js_is_j$ is equivalent to $s_i\tau_j s_i=s_j \tau_is_j$, for all $|i-j| = 1$.
We will refer to the relations \ $\nu_i\nu_j\nu_i = \nu_j\nu_i\nu_j$, $\nu_is_j\nu_i = \nu_js_i\nu_j$, and $\nu_i\tau_j\nu_i = \nu_j\tau_i\nu_j$, for $|i-j| = 1$, collectively as \textit{detour relations}. Finally, we refer to the last set of relations in Definition~\ref{def:VSTM} as the \textit{commuting relations}.

There is an obvious group homomorphism that associates to each virtual singular twin braid in $VST_n$ a permutation in $S_n$, via the following map:
\[ \pi: VST_n \to S_n, \,\, \pi(s_i) =\pi(\tau_i) =\pi(\nu_i) =(i, i+1), \,\,\, \text{for all} \,\, \, 1 \leq i \leq n-1. \]
We extend $\pi$ to all elements in $VST_n$ such that $\pi$ is a group homomorphism. Since the set $\{(i, i+1) \mid 1 \le i \le n-1\}$ generates $S_n$, the map $\pi$ is surjective. We may regard $S_n$ as a subgroup of $VST_n$ generated by $\{\nu_i \,  \big{|} \,1 \le i \le n-1\}$, by identifing $\nu_i$ with the transposition $(i, i+1)$. The virtual singular twin braids whose associated permutation is the identity permutation form a normal subgroup of $VST_n$.

\begin{definition} \label{Virtual Singular Pure Twin Group}
We call the kernel of $\pi$ the \textit{virtual singular pure twin group}, and denote it by $VSPT_n$. We refer to an element of $VSPT_n$ as a virtual singular pure twin braid.
\end{definition}

Clearly, $VSPT_n$ is a normal subgroup of $VST_n$ and $VST_n/VSPT_n \cong S_n$. 

In Section \ref{sec:different_presentations} we will provide three alternative presentations for the virtual singular twin monoid $VSTM_n$.

\section{On the Extensions of the Representations $\eta_1$ and $\eta_2$ of $T_n$ to $VST_n$} \label{sec:rep}

In this section, we examine all possible homogeneous $2$-local representations of the virtual singular twin group $VST_n$, for $n\geq 3$, that extend the representations $\eta_1$ and $\eta_2$ of $T_n$, as introduced in Definitions \ref{defeta1} and  \ref{defeta2}. We also investigate some properties of these extended representations. Throughout this section, we assume $n\geq 3$.

\begin{theorem} \label{Theta1}
Let $\eta'_1: VST_n \rightarrow \mathrm{GL}_n(\mathbb{Z}[t^{\pm 1}])$, where $t$ is an indeterminate, be a homogeneous $2$-local representation of $VST_n$ that extends the representation $\eta_1$ given in Definition \ref{defeta1}. Then, the action of $\eta'_1$ on the generators of $VST_n$ is given as follows:
\begin{align*}
\eta_1'(s_i) &= \left( \begin{array}{c|@{}c|c@{}}
   \begin{matrix}
     I_{i-1} 
   \end{matrix} 
      & 0 & 0 \\
      \hline
    0 &\hspace{0.2cm} \begin{matrix}
   	1-t & t\\
   	2-t & t-1\\
\end{matrix}  & 0  \\
\hline
0 & 0 & I_{n-i-1}
\end{array} \right), \\
\\
\eta_1'(\tau_i) & =I_n,
\\
\eta_1'(\nu_i) &=  \left( \begin{array}{c|@{}c|c@{}}
   \begin{matrix}
     I_{i-1} 
   \end{matrix} 
      & 0 & 0 \\
      \hline
    0 &\hspace{0.2cm} \begin{matrix}
   	0 & v(t)\\
   	\dfrac{1}{v(t)} & 0\\
\end{matrix}  & 0  \\
\hline
0 & 0 & I_{n-i-1}
\end{array} \right),
\end{align*} 
for $1\leq i \leq n-1$, where $v(t) \in \mathbb{Z}[t^{\pm 1}]$ with $v(t)$  invertible in $\mathbb{Z}[t^{\pm 1}]$.
\end{theorem}

\begin{proof}
Since $\eta_1'$ extends $\eta_1$, it follows that $$\eta_1'(s_i)= \left( \begin{array}{c|@{}c|c@{}}
   \begin{matrix}
     I_{i-1} 
   \end{matrix} 
      & 0 & 0 \\
      \hline
    0 &\hspace{0.2cm} \begin{matrix}
   	1-t & t\\
   	2-t & t-1\\
\end{matrix}  & 0  \\
\hline
0 & 0 & I_{n-i-1}
\end{array} \right) \text{ for  } 1\leq i \leq n-1.$$
Since $\eta_1'$ is a homogeneous $2$-local representation, we may set $$\eta_1'(\tau_i)= \left( \begin{array}{c|@{}c|c@{}}
   \begin{matrix}
     I_{i-1} 
   \end{matrix} 
      & 0 & 0 \\
      \hline
    0 &\hspace{0.2cm} \begin{matrix}
   	w(t) & x(t)\\
   	y(t) & z(t)\\
\end{matrix}  & 0  \\
\hline
0 & 0 & I_{n-i-1}
\end{array} \right) \text{ and } \eta_1'(\nu_i)= \left( \begin{array}{c|@{}c|c@{}}
   \begin{matrix}
     I_{i-1} 
   \end{matrix} 
      & 0 & 0 \\
      \hline
    0 &\hspace{0.2cm} \begin{matrix}
   	u(t) & v(t)\\
   	r(t) & s(t)\\
\end{matrix}  & 0  \\
\hline
0 & 0 & I_{n-i-1}
\end{array} \right)$$
for $1\leq i \leq n-1$, where $w(t), x(t), y(t), z(t), u(t), v(t), r(t), s(t) \in \mathbb{Z}[t^{\pm 1}]$. For $\eta'_1$ to be a representation of $VST_n$, it must preserve the defining relations in $VST_n$.  Note that since we are dealing with homogeneous $2$-local representations, it is enough to consider the following relations among the generators of $VST_n$:
\begin{equation} \label{eta_prime1_firstrel}
\nu_1^2=1,
\end{equation}
\begin{equation}
s_1\tau_1=\tau_1s_1,
\end{equation}
\begin{equation}
s_1s_2\tau_1=\tau_2s_1s_2,
\end{equation}
\begin{equation}
s_2s_1\tau_2=\tau_1s_2s_1,
\end{equation}
\begin{equation}
\nu_1\nu_2\nu_1=\nu_2\nu_1\nu_2,
\end{equation}
\begin{equation}
\nu_1s_2\nu_1=\nu_2s_1\nu_2,
\end{equation}
\begin{equation}\label{eta_prime1_lastrel}
\nu_1\tau_2\nu_1=\nu_2\tau_1\nu_2.
\end{equation}
By imposing that $\eta'_1$ preserves these relations, we derive a system of equations involving eight unknowns. Solving this system through direct algebraic methods yields the desired results.
\end{proof}

\begin{theorem} \label{eta1irr}
The representation $\eta'_1: VST_n \rightarrow \mathrm{GL}_n(\mathbb{Z}[t^{\pm 1}])$, where $t$ is an indeterminate, is irreducible if and only if $v(t)\neq 1$.
\end{theorem}

\begin{proof}
For the necessary condition suppose that $v(t)=1$. We can directly see that the column vector $(1,1,\ldots, 1)^T$ is invariant under $\eta_1'(s_i), \eta_1'(\tau_i),$ and $\eta_1'(\nu_i)$ for all $1\leq i \leq n-1$. Hence, $(1,1,\ldots, 1)^T$ is invariant under $\eta'_1$, and then $\eta'_1$ is reducible.

For the sufficient condition, suppose that $v(t) \neq 1$.  We introduce an equivalent representation of $\eta_1'$, namely $\eta_1''$, in the following manner. Let 
$$D = \mathrm{diag}\left(v(t)^{n-1}, v(t)^{n-2}, \ldots, v(t), 1\right),$$ 
where $\mathrm{diag}(a_1, a_2, \ldots, a_n)$ denotes the $n \times n$ diagonal matrix with diagonal entries $a_1, a_2, \ldots, a_n$. Define a new representation $\eta_1''$ of $VST_n$ by 
$$\eta_1'' = D^{-1}\eta_1'D.$$ 
By its definition, we can directly see that $\eta_1''$ is equivalent to $\eta'_1$, and we obtain that the action of $\eta_1''$ on the generators of $VST_n$ is given as follows:
\begin{align*}
\eta_1''(s_i) &= 
\left( \begin{array}{c|@{}c|c@{}}
   \begin{matrix} I_{i-1} \end{matrix} & 0 & 0 \\
   \hline
   0 & \hspace{0.2cm} \begin{matrix} 1-t & \dfrac{t}{v(t)}\\ (2-t)v(t) & t-1 \end{matrix} & 0 \\
   \hline
   0 & 0 & I_{n-i-1} 
\end{array} \right), \\
\eta_1''(\tau_i) &= I_n, \\ 
\eta_1''(\nu_i) &=
\left( \begin{array}{c|@{}c|c@{}}
   \begin{matrix} I_{i-1} \end{matrix} & 0 & 0 \\
   \hline
   0 & \hspace{0.2cm} \begin{matrix} 0 & 1\\ 1 & 0 \end{matrix} & 0 \\
   \hline
   0 & 0 & I_{n-i-1} 
\end{array} \right).
\end{align*}
Assume now, for the sake of contradiction, that $\eta_1'$ is reducible. Then, its equivalent representation $\eta_1''$ must also be reducible. Hence, there exists a nontrivial subspace $U \subset \mathbb{Z}^n[t^{\pm 1}]$ that is invariant under the action of $\eta_1''$. Choose a nonzero vector $x = (x_1, x_2, \ldots, x_n)^T \in U$. We observe that
$$\eta_1''(\nu_i)(x) - x = (x_{i+1} - x_i)(e_i - e_{i+1})\in U$$
for all $1 \leq i \leq n-1$, where $e_1, e_2, \ldots, e_n$ denote the canonical basis vectors of $\mathbb{Z}^n[t^{\pm 1}]$. We can select $x$ such that there exists $1 \leq j \leq n-1$ with $x_j \neq x_{j+1}$; otherwise $U$ would be generated by $(1, 1, \ldots, 1)^T$, contradicting $v(t) \neq 1$, since $U$ is invariant under $\eta_1''(s_i)$. Hence, there exists $1 \leq j \leq n-1$ such that $e_j - e_{j+1} \in U$. Moreover,
$$
\eta_1''(\nu_{j+1})(e_j - e_{j+1})-(e_j - e_{j+1}) = e_{j+1} - e_{j+2} \in U
$$ 
and
$$\eta_1''(\nu_{j-1})(e_j - e_{j+1})-(e_j - e_{j+1}) = e_{j-1} - e_j \in U.
$$
By continuing this recursive process, we get that
\begin{equation} \label{eq:diffs}
e_i - e_{i+1} \in U \text{ for all } 1 \leq i \leq n-1. \tag{E}
\end{equation}
Now, we can see that $$\eta_1''(s_1)(e_1-e_2)-\left( 1 - t - \dfrac{t}{v(t)} \right)(e_1-e_2)=\left( (2-t)v(t)+2 -2t - \dfrac{t}{v(t)} \right)e_2\in U.$$ But $e_2$ cannot be an element of $U$, since otherwise we get $e_i \in U$ for all $1\leq i \leq n$ by Equation (\ref{eq:diffs}), and so we get that $U=\mathbb{Z}^n[t^{\pm 1}]$, which is a contradiction. Thus, we have $$(2-t)v(t)+2 -2t - \dfrac{t}{v(t)}=0.$$ Solving this equation to obtain $v(t)$ yields that $v(t)=-1$ or $v(t)=\dfrac{t}{2-t}$. But $\dfrac{t}{2-t} \notin \mathbb{Z}[t^{\pm 1}]$ and so $v(t)=-1$. Now, we can see in our case that 
$$\eta_1''(s_1)(e_2 - e_3)+t(e_1-e_2)=-e_2-e_3\in U.$$
But $e_2-e_3\in U$ by Equation \eqref{eq:diffs}, which implies that $e_2\in U$, yielding a contradiction. This completes the proof that $\eta_1''$ is irreducible in this case, and so $\eta_1'$ is also irreducible.
\end{proof}

\begin{proposition}
The representation $\eta'_1: VST_n \rightarrow \mathrm{GL}_n(\mathbb{Z}[t^{\pm 1}])$, where $t$ is an indeterminate, is unfaithful for all $n\geq 3$.
\end{proposition}

\begin{proof}
The proof is straightforward as the elements $\tau_i$, for all $1\leq i \leq n-1,$ map to the identity matrix $I_n$.
\end{proof}

\begin{theorem} \label{Theta2}
Let $\eta'_2: VST_n \rightarrow \mathrm{GL}_n(\mathbb{Z}[t^{\pm 1}])$, where $t$ is an indeterminate, be a homogeneous $2$-local representation of $VST_n$ that extends the representation $\eta_2$ given in Definition \ref{defeta2}. Then the action of $\eta'_2$ on the generators of $VST_n$ is defined as follows:
\begin{align*}
\eta'_2(s_i) &= \left( \begin{array}{c|@{}c|c@{}}
   \begin{matrix}
     I_{i-1} 
   \end{matrix} 
      & 0 & 0 \\
      \hline
    0 &\hspace{0.2cm} \begin{matrix}
   	0 & f(t)\\
   	\dfrac{1}{f(t)} & 0\\
\end{matrix}  & 0  \\
\hline
0 & 0 & I_{n-i-1}
\end{array} \right), \\
\eta'_2(\tau_i) &= \left( \begin{array}{c|@{}c|c@{}}
   \begin{matrix}
     I_{i-1} 
   \end{matrix} 
      & 0 & 0 \\
      \hline
    0 &\hspace{0.2cm} \begin{matrix}
   	w(t) & f^2(t)y(t)\\
   	y(t) & w(t)\\
\end{matrix}  & 0  \\
\hline
0 & 0 & I_{n-i-1}
\end{array} \right), \\
\eta'_2(\nu_i) &=  \left( \begin{array}{c|@{}c|c@{}}
   \begin{matrix}
     I_{i-1} 
   \end{matrix} 
      & 0 & 0 \\
      \hline
    0 &\hspace{0.2cm} \begin{matrix}
   	0 & v(t)\\
   	\dfrac{1}{v(t)} & 0\\
\end{matrix}  & 0  \\
\hline
0 & 0 & I_{n-i-1}
\end{array} \right),
\end{align*}  
for $1\leq i \leq n-1$, where $f(t), w(t), y(t), v(t) \in \mathbb{Z}[t^{\pm 1}]$, $f(t)$ and $v(t)$ are invertible in $\mathbb{Z}[t^{\pm 1}]$. 
\end{theorem}

\begin{proof}
The proof is similar to the proof of Theorem \ref{Theta1}.
\end{proof}

\begin{theorem} \label{eta2irr}
The representation $\eta'_2: VST_n \rightarrow \mathrm{GL}_n(\mathbb{Z}[t^{\pm 1}])$, where $t$ is an indeterminate, is irreducible if and only if $f(t)\neq v(t)$ or $w(t)+\dfrac{f^2(t)y(t)}{v(t)} \neq 1$ or $v(t)y(t)+w(t)\neq 1$.
\end{theorem}

\begin{proof}
As in the proof of Theorem \ref{eta1irr}, we introduce an equivalent representation of $\eta'_2$, denoted by $\eta''_2$, defined as follows. Let 
$$C = \mathrm{diag}\left(v(t)^{n-1}, v(t)^{n-2}, \ldots, v(t), 1\right),$$
and set 
$$\eta_2'' = C^{-1}\eta_2' C.$$
Note that we may describe the action of $\eta_2''$ on the generators of $VST_n$ as given below:
$$\eta_2''(s_i)=
\left(
\begin{array}{c|@{}c|c@{}}
\begin{matrix} I_{i-1} \end{matrix} & 0 & 0\\
\hline
0 & \hspace{0.2cm}\begin{matrix}
0 & \dfrac{f(t)}{v(t)}\\[4pt]
\dfrac{v(t)}{f(t)} & 0
\end{matrix} & 0\\
\hline
0 & 0 & I_{n-i-1}
\end{array}
\right),$$
$$\eta_2''(\tau_i)=
\left(
\begin{array}{c|@{}c|c@{}}
\begin{matrix} I_{i-1} \end{matrix} & 0 & 0\\
\hline
0 & \hspace{0.2cm}\begin{matrix}
w(t) & \dfrac{f^2(t)y(t)}{v(t)}\\[4pt]
v(t)y(t) & w(t)
\end{matrix} & 0\\
\hline
0 & 0 & I_{n-i-1}
\end{array}
\right),$$
$$\eta_2''(\nu_i)=
\left(
\begin{array}{c|@{}c|c@{}}
\begin{matrix} I_{i-1} \end{matrix} & 0 & 0\\
\hline
0 & \hspace{0.2cm}\begin{matrix}
0 & 1\\
1 & 0
\end{matrix} & 0\\
\hline
0 & 0 & I_{n-i-1}
\end{array}
\right)$$
for all $1 \leq i \leq n-1$.

\vspace{0.1cm}

For the necessary condition, suppose that $f(t)=v(t)$, $w(t)+\dfrac{f^2(t)y(t)}{v(t)}= 1,$ 
and $v(t)y(t)+w(t)= 1.$ Then, we can directly observe that the column vector $(1,1,\ldots,1)^T$ is invariant under $\eta_2''(s_i)$, $\eta_2''(\tau_i),$ and $\eta_2''(\nu_i)$, for all $1\leq i \leq n-1$. Hence $\eta''_2$ is reducible, which implies that $\eta'_2$ is reducible as well.


Now, for the sufficient condition, suppose that $f(t)\neq v(t)$ or $w(t)+\dfrac{f^2(t)y(t)}{v(t)} \neq 1$ or $v(t)y(t)+w(t)\neq 1.$
Assume, for the sake of contradiction, that $\eta_2'$ is reducible. Then, its equivalent representation $\eta_2''$, introduced above, must also be reducible. It follows that there exists a nontrivial subspace $U \subset \mathbb{Z}^n[t^{\pm 1}]$ that is invariant under the action of $\eta_2''$. Repeating the same reasoning as in the proof of Theorem \ref{eta1irr}, and under our condition, we get that the equation (\ref{eq:diffs}) holds here as well. Recall that the equation (\ref{eq:diffs}) is given by the following:
\begin{equation} \label{eq:diffs}
e_i - e_{i+1} \in U \text{ for all } 1 \leq i \leq n-1. \tag{E}
\end{equation}
We now consider each case separately.
\begin{itemize}
\item[(a)] Suppose $f(t)\neq v(t)$. In this case, we have
$$\eta_2''(s_1)(e_2-e_3)-\dfrac{f(t)}{v(t)}(e_1-e_2)-(e_2-e_3) = \left(\dfrac{f(t)}{v(t)}-1\right)e_2 \in U.$$
Since $\dfrac{f(t)}{v(t)}-1\neq 0$, it follows that $e_2 \in U$. By Equation (\ref{eq:diffs}), this implies that $e_i \in U$ for all $1 \leq i \leq n-1$, which is a contradiction.
\item[(b)] Suppose $w(t)+\dfrac{f^2(t)y(t)}{v(t)} \neq 1$. Here, we have
$$
\eta_2''(\tau_1)(e_2-e_3)-\dfrac{f^2(t)y(t)}{v(t)}(e_1-e_2)-(e_2-e_3)= \left(w(t)+\dfrac{f^2(t)y(t)}{v(t)}-1\right)e_2 \in U.$$

Since $w(t)+\dfrac{f^2(t)y(t)}{v(t)}-1\neq 0$, we again get $e_2 \in U$, leading to the same contradiction as above.
\item[(c)] Suppose $v(t)y(t)+w(t)\neq 1$. Then,
$$\eta_2''(\tau_2)(e_1-e_2)-(e_1-e_2)-v(t)y(t)(e_2-e_3) = (-w(t)-v(t)y(t)+1)e_2 \in U.$$
Again, as $-w(t)-v(t)y(t)+1\neq 0$, we get $e_2 \in U$, giving the same contradiction.
\end{itemize}

In all cases, we arrive at a contradiction, showing that no nontrivial invariant subspace $U$ exists. Therefore, $\eta_2''$ (and equivalently $\eta_2'$) is irreducible in this case, which completes the proof.
\end{proof}

\begin{proposition}
The representation $\eta'_2: VST_n \rightarrow \mathrm{GL}_n(\mathbb{Z}[t^{\pm 1}])$, where $t$ is an indeterminate, is unfaithful for all $n\geq 3$.
\end{proposition}

\begin{proof}
Since Theorem \ref{eta2fai} demonstrates that the restriction of $\eta_2'$ to the twin group $T_n$ is unfaithful, the claim follows immediately.
\end{proof}

In summary, for all $n \geq 3$, we have classified the homogeneous $2$-local representations of $VST_n$ that extend the representations $\eta_1$ and $\eta_2$ of $T_n$, and examined their properties with respect to irreducibility and faithfulness. This naturally leads to the following question.

\begin{question}
Is it possible to construct non-local representations of $VST_n$ that extend the representations $\eta_1$ and $\eta_2$ of $T_n$?
\end{question}

\section{Complex Homogeneous Local Representations of the Virtual Singular Twin Group} \label{sec:complex_rep}

In this section, we determine all complex homogeneous $2$-local representations of $VST_n$ for all $n\geq 3$. 

\begin{theorem} \label{TheoSVTn}
Let $n\geq 3$ and let $\Upsilon: VST_n \rightarrow \mathrm{GL}_n(\mathbb{C})$ be a complex homogeneous $2$-local representation of $VST_n$. Then $\Upsilon$ is equivalent to one of the following six representations $\Upsilon_j, 1\leq j \leq 6$, where 
$$\Upsilon_j(s_i) =\left( \begin{array}{c|@{}c|c@{}}
   \begin{matrix}
     I_{i-1} 
   \end{matrix} 
      & 0 & 0 \\
      \hline
    0 &\hspace{0.2cm} S_j  & 0  \\
\hline
0 & 0 & I_{n-i-1}
\end{array} \right),$$ $$\Upsilon_j(\tau_i) =\left( \begin{array}{c|@{}c|c@{}}
   \begin{matrix}
     I_{i-1} 
   \end{matrix} 
      & 0 & 0 \\
      \hline
    0 &\hspace{0.2cm} T_j  & 0  \\
\hline
0 & 0 & I_{n-i-1}
\end{array} \right),$$ and $$\Upsilon_j(\nu_i) =\left( \begin{array}{c|@{}c|c@{}}
   \begin{matrix}
     I_{i-1} 
   \end{matrix} 
      & 0 & 0 \\
      \hline
    0 &\hspace{0.2cm} V_j  & 0  \\
\hline
0 & 0 & I_{n-i-1}
\end{array} \right).\vspace*{0.1cm}$$
\begin{itemize}
\item[(1)] $\Upsilon_1: VST_n \rightarrow \mathrm{GL}_n(\mathbb{C})$ such that 
$$S_1=\left( \begin{array}{@{}c@{}}
\begin{matrix}
   		0 & b\\
   		\dfrac{1}{b} & 0
   		\end{matrix}
\end{array} \right), T_1=\left( \begin{array}{@{}c@{}}
\begin{matrix}
   		x & y\\
   		\dfrac{y}{b^2} & x
   		\end{matrix}
\end{array} \right), \text{and } V_1=\left( \begin{array}{@{}c@{}}
\begin{matrix}
   		0 & v\\
   		\dfrac{1}{v} & 0
   		\end{matrix}
\end{array} \right),$$ where $b,x,y,v \in \mathbb{C}, b\neq 0, x^2-\dfrac{y^2}{b^2}\neq 0 ,v\neq 0$.\\
\item[(2)] $\Upsilon_2: VST_n \rightarrow \mathrm{GL}_n(\mathbb{C})$ such that 
$$S_2=\left( \begin{array}{@{}c@{}}
\begin{matrix}
   		-\sqrt{1-bc} & b\\
   		c & \sqrt{1-bc}
   		\end{matrix}
\end{array} \right), T_2=\left( \begin{array}{@{}c@{}}
\begin{matrix}
   		1 & 0\\
   		0 & 1
   		\end{matrix}
\end{array} \right), \text{and } V_2=\left( \begin{array}{@{}c@{}}
\begin{matrix}
   		0 & v\\
   		\dfrac{1}{v} & 0
   		\end{matrix}
\end{array} \right),$$ where $b,c,v \in \mathbb{C}, v\neq 0$.\\
\item[(3)] $\Upsilon_3: VST_n \rightarrow \mathrm{GL}_n(\mathbb{C})$ such that 
$$S_3=\left( \begin{array}{@{}c@{}}
\begin{matrix}
   		\sqrt{1-bc} & b\\
   		c & -\sqrt{1-bc}
   		\end{matrix}
\end{array} \right), T_3=\left( \begin{array}{@{}c@{}}
\begin{matrix}
   		1 & 0\\
   		0 & 1
   		\end{matrix}
\end{array} \right), \text{and } V_3=\left( \begin{array}{@{}c@{}}
\begin{matrix}
   		0 & v\\
   		\dfrac{1}{v} & 0
   		\end{matrix}
\end{array} \right),$$ where $b,c,v \in \mathbb{C}, v\neq 0$.\\
\item[(4)] $\Upsilon_4: VST_n \rightarrow \mathrm{GL}_n(\mathbb{C})$ such that 
$$S_4=\left( \begin{array}{@{}c@{}}
\begin{matrix}
   		-1 & 0\\
   		0 & -1
   		\end{matrix}
\end{array} \right), T_4=\left( \begin{array}{@{}c@{}}
\begin{matrix}
   		1 & 0\\
   		0 & 1
   		\end{matrix}
\end{array} \right), \text{and } V_4=\left( \begin{array}{@{}c@{}}
\begin{matrix}
   		0 & v\\
   		\dfrac{1}{v} & 0
   		\end{matrix}
\end{array} \right),$$ where $v \in \mathbb{C}, v\neq 0$.\\
\item[(5)] $\Upsilon_5: VST_n \rightarrow \mathrm{GL}_n(\mathbb{C})$ such that 
$$S_5=\left( \begin{array}{@{}c@{}}
\begin{matrix}
   		1 & 0\\
   		0 & 1
   		\end{matrix}
\end{array} \right), T_5=\left( \begin{array}{@{}c@{}}
\begin{matrix}
   		1 & 0\\
   		0 & 1
   		\end{matrix}
\end{array} \right), \text{and } V_5=\left( \begin{array}{@{}c@{}}
\begin{matrix}
   		0 & v\\
   		\dfrac{1}{v} & 0
   		\end{matrix}
\end{array} \right),$$ where $v \in \mathbb{C}, v\neq 0$.\\
\item[(6)] $\Upsilon_6: VST_n \rightarrow \mathrm{GL}_n(\mathbb{C})$ such that 
$$S_6=\left( \begin{array}{@{}c@{}}
\begin{matrix}
   		1 & 0\\
   		0 & 1
   		\end{matrix}
\end{array} \right), T_6=\left( \begin{array}{@{}c@{}}
\begin{matrix}
   		1 & 0\\
   		0 & 1
   		\end{matrix}
\end{array} \right), \text{and } V_6=\left( \begin{array}{@{}c@{}}
\begin{matrix}
   		1 & 0\\
   		0 & 1
   		\end{matrix}
\end{array} \right).$$
\end{itemize}
\end{theorem}

\begin{proof}
Set
$$\Upsilon(s_i) =\left( \begin{array}{c|@{}c|c@{}}
   \begin{matrix}
     I_{i-1} 
   \end{matrix} 
      & 0 & 0 \\
      \hline
    0 &\hspace{0.2cm} S  & 0  \\
\hline
0 & 0 & I_{n-i-1}
\end{array} \right),$$ $$\Upsilon(\tau_i) =\left( \begin{array}{c|@{}c|c@{}}
   \begin{matrix}
     I_{i-1} 
   \end{matrix} 
      & 0 & 0 \\
      \hline
    0 &\hspace{0.2cm} T  & 0  \\
\hline
0 & 0 & I_{n-i-1}
\end{array} \right),$$ and $$\Upsilon(\nu_i) =\left( \begin{array}{c|@{}c|c@{}}
   \begin{matrix}
     I_{i-1} 
   \end{matrix} 
      & 0 & 0 \\
      \hline
    0 &\hspace{0.2cm} V  & 0  \\
\hline
0 & 0 & I_{n-i-1}
\end{array} \right),$$
where $$S=\left( \begin{array}{@{}c@{}}
  \begin{matrix}
   		a & b\\
   		c & d
   		\end{matrix}
\end{array} \right), T=\left( \begin{array}{@{}c@{}}
  \begin{matrix}
   		x & y\\
   		z & t
   		\end{matrix}
\end{array} \right), \text{ and }  V=\left( \begin{array}{@{}c@{}}
  \begin{matrix}
   		u & v\\
   		r & s
   		\end{matrix}
\end{array} \right),$$
with $a,b,c,d,x,y,z,t,u,v,r,s \in \mathbb{C}, ad-bc \neq 0, xt-yz\neq0,$ and $us-rv \neq 0$. To determine the matrices $S, T$ and $V$, we impose the condition that $\Upsilon$ preserves the defining relations  for $VST_n$. As in the proof of Theorem~\ref{Theta1}, since $\Upsilon$ is a complex homogeneous $2$-local representation, it suffices to consider only the relations~\eqref{eta_prime1_firstrel} through~\eqref{eta_prime1_lastrel}, together with $s_1^2 = 1$, among the generators of $VST_n$, as all other relations yield analogous equations.
Imposing that $\Upsilon$ is a representation leads to a system of equations in twelve unknowns. Solving this system using direct algebraic methods produces the desired results.
\end{proof}


We conclude this section by posing the following natural question.

\begin{question}
Let \(n\geq 3\) and let \(\Upsilon: VST_n \rightarrow \mathrm{GL}_n(\mathbb{C})\) be a complex homogeneous \(2\)-local representation of \(VST_n\). Under what conditions is \(\Upsilon\)  irreducible?
\end{question}

\section{Other Presentations for the Virtual Singular Twin Monoid} \label{sec:different_presentations}

The scope of this section is to provide additional presentations for the virtual singular twin monoid $VSTM_n$.

\subsection{A Reduced Presentation for $VSTM_n$}

In~\cite[Section 3] {KL1}, L. H. Kauffman and S. Lambropoulou provided a reduced presentation for the virtual braid group. Similarly, C. Caprau et al. provided a reduced presentation for the virtual singular braid monoid in~\cite[Section 4]{Caprau}. Inspired by these results, in this section we give a reduced presentation for the virtual singular twin monoid on $n$ strands, $VSTM_n$. This presentation uses fewer generators, specifically $s_1$ and $\tau_1$, together with all the virtual generators $\nu_1,\ldots, \nu_{n-1}$, and assumes the following defining relations for the elements $s_{i+1}$ and $\tau_{i+1}$, for all $1\leq i \leq n-2$:
\begin{eqnarray}
s_{i+1}   & := & (\nu_i\ldots \nu_2\nu_1)(\nu_{i+1}\ldots \nu_3\nu_2)s_1  (\nu_2\nu_3\ldots \nu_{i+1})(\nu_1\nu_2\ldots \nu_i)  \label{A14} \\  
\tau_{i+1} & := & (\nu_i\ldots \nu_2\nu_1)(\nu_{i+1}\ldots \nu_3 \nu_2)\tau_1 (\nu_2 \nu_3\ldots \nu_{i+1})(\nu_1\nu_2\ldots \nu_i) \label{A15} 
\end{eqnarray}
 As shown in Fig.~\ref{fig:DRtau}, the defining relations are the braid form versions of the detour move. In other words, we detour the generators $s_{i+1}$ and $\tau_{i+1}$ to the left side of the braid using the strands $1, 2, \dots, i$.

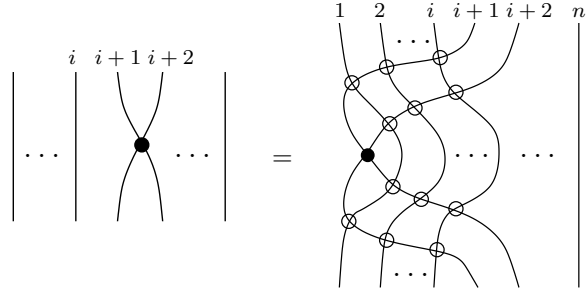
\begin{figure}[h!]
\[ \raisebox{-35pt}{\includegraphics[height=0.8in]{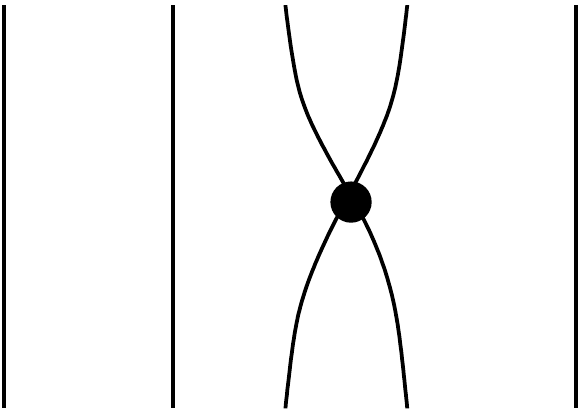}}
 \put(-77, -10){\fontsize{12}{8}$\dots$}
  \put(-20, -10){\fontsize{12}{8}$\dots$}
  \put(-60, 25){\fontsize{8}{8}$i$}
  \put(-50, 25){\fontsize{8}{8}$i+1$}
 \put(-30, 25){\fontsize{8}{8}$i+2$} \hspace{0.6cm}
 \raisebox{-13pt}{=} \hspace{0.6cm}
 \raisebox{-60pt}{\includegraphics[height=1.4in]{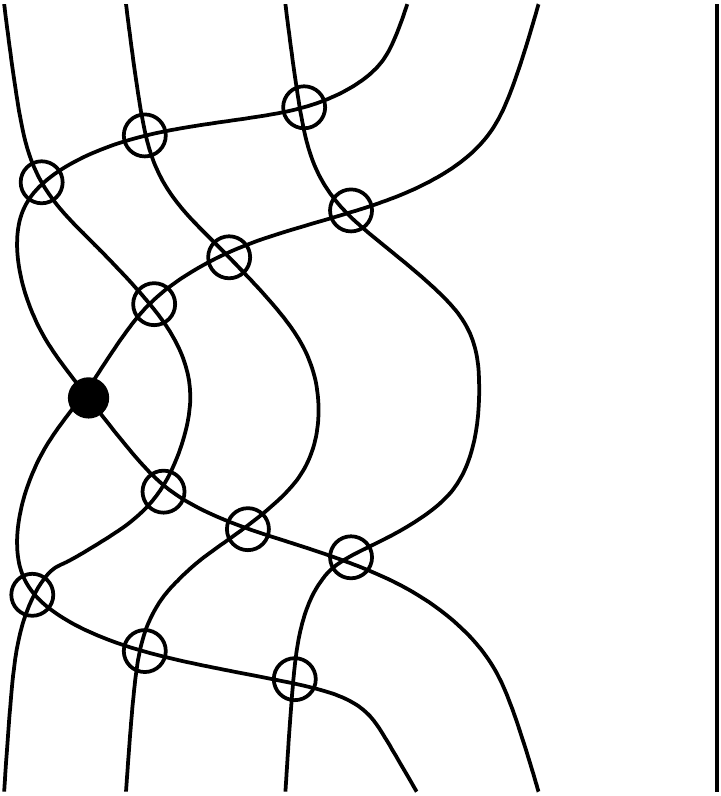}} 
  \put(-23, -10){\fontsize{12}{8}$\dots$}
    \put(-48, -10){\fontsize{12}{8}$\dots$}
     \put(-93, 43){\fontsize{8}{8}$1$}
      \put(-78, 43){\fontsize{8}{8}$2$}
       \put(-71, -55){\fontsize{12}{8}$\dots$}
          \put(-70, 33){\fontsize{12}{8}$\dots$}
    \put(-58, 43){\fontsize{8}{8}$i$}
  \put(-48, 43){\fontsize{8}{8}$i+1$}
 \put(-28, 43){\fontsize{8}{8}$i+2$}
  \put(-3, 43){\fontsize{8}{8}$n$}
 \]
 \caption{Detouring the generator $\tau_{i+1}$}\label{fig:DRtau} 
\end{figure} 

As shown in Fig.~\ref{fig:braid detour-moves}, we can detour to the front of the braid any portion of a given braid, where all of the new crossings that are created are virtual. For this reason, in the reduced presentation for $VSTM_n$, the relations involving the elements $s_i$ and $\tau_i$ will be imposed to occur between the first strands of a braid. 

\begin{figure}[h!]
\[ \raisebox{-55pt}{\includegraphics[height=1.6in]{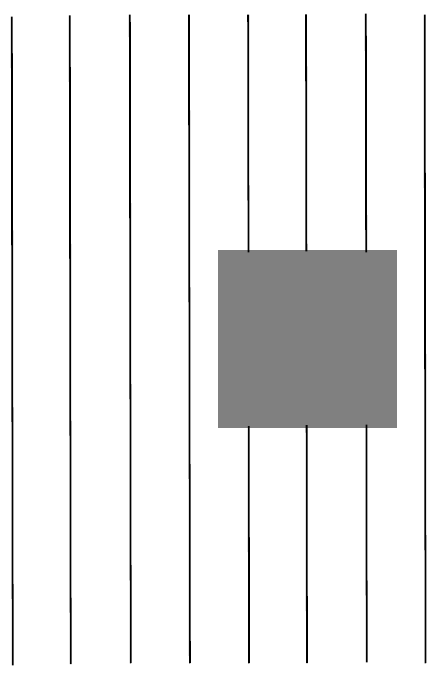}}\,\, \longleftrightarrow \,\, \raisebox{-55pt}{\includegraphics[height=1.6in]{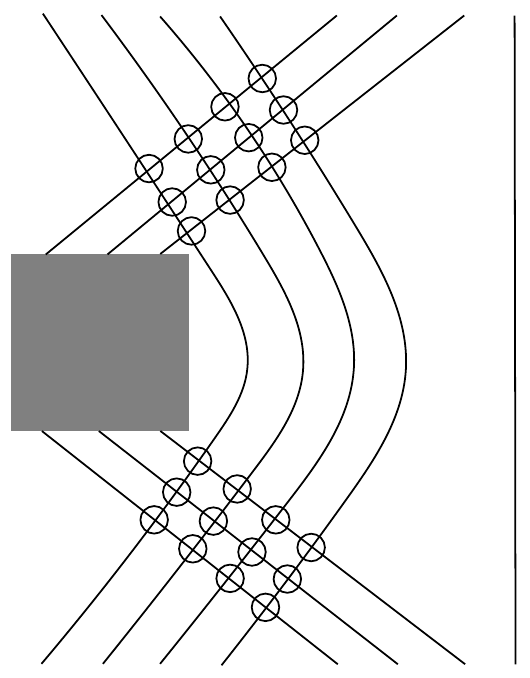}}
\]
\caption{Detouring several strands in a braid}\label{fig:braid detour-moves}
\end{figure} 

We remark that the relations $\nu_is_j \nu_i = \nu_j s_i \nu_j$ and $\nu_i\tau_j \nu_i = \nu_j \tau_i \nu_j$ for $|i-j|=1$ are not needed in the reduced presentation for $VSB_n$, since they are implicitelly used in the defining relations \eqref{A14} and \eqref{A15}.

\begin{theorem} \label{thm:reduced}
The virtual singular twin monoid $VSTM_n$ has the following reduced presentation with generators $\{s_1, \tau_1, \nu_1,\ldots, \nu_{n-1}\},$ 
and relations:
\begin{align}
\nu_i \nu_j \nu_i &=  \nu_j \nu_i \nu_j, && |i-j|=1          \label{eqs36}\\
\nu_i \nu_j      &= \nu_j \nu_i             &&  |i-j| \geq 2    \label{eqs37} \\ 
{\nu_i}^2      &= 1 \,\,  \text{and}\ \,\, s_1^2 = 1                   && 1 \leq i \leq n-1   \label{eqs38}\\
s_1\tau_1 &=  \tau_1 s_1  \,\,   \label{eqs39} \\ 
\tau_1 \nu_i  &= \nu_i \tau_1\,\,\, \text{and}\,\, \,s_1\nu_i = \nu_is_1    &&\, i \geq 3    \label{eqs40} \\
\tau_1( \nu_1 \nu_2 s_1 \nu_2  \nu_1 ) s_1  &= ( \nu_1 \nu_2 s_1 \nu_2 \nu_1 )s_1 (\nu_1 \nu_2 \tau_1 \nu_2 \nu_1 ) \label{eqs41} \\ 
s_1 (\nu_2 \nu_3 \nu_1 \nu_2 s_1 \nu_2 \nu_1 \nu_3 \nu_2) &= (\nu_2 \nu_3 \nu_1 \nu_2 s_1 \nu_2 \nu_1 \nu_3 \nu_2)  s_1 \label{eqs42} \\ 
\tau_1 (\nu_2 \nu_3 \nu_1 \nu_2 s_1 \nu_2 \nu_1 \nu_3 \nu_2) &= (\nu_2 \nu_3 \nu_1 \nu_2 s_1 \nu_2 \nu_1 \nu_3 \nu_2) \tau_1 \label{eqs43}\\
\tau_1 (\nu_2 \nu_3 \nu_1 \nu_2 \tau_1 \nu_2 \nu_1 \nu_3 \nu_2) &= (\nu_2 \nu_3 \nu_1 \nu_2 \tau_1 \nu_2 \nu_1 \nu_3 \nu_2) \tau_1 \label{eqs44} 
\end{align}
\end{theorem}

Note that in the reduced presentation for $VSTM_n$ we kept all of the original relations involving only virtual generators. In addition, we have kept the relations involving $s_i$ or $\tau_i$ that can be represented on the left side of the braid. For convenience, we will call these the \textit{base cases} of the original relations. For example, the base case for the commuting relations $s_i s_j = s_j s_i$ is the relation $s_1s_3 = s_3s_1$, which by the defining relation~\eqref{A14} is equivalent to the relation~\eqref{eqs42}. Similarly, from the commuting relations $\tau_i s_j = s_j \tau_i$ and $\tau_i \tau_j = \tau_j \tau_i$ (with $|i -j| \geq 2$) we kept only the relations $\tau_1 s_3 = s_3 \tau_1$ and, respectively, $\tau_1 \tau_3 = \tau _3 \tau_1$, which are represented by the relations~\eqref{eqs43} and~\eqref{eqs44}, respectively. We will show that all of the other commuting relations follow from their corresponding base case relations and the virtual relations. In addition, we remark that relation~\eqref{eqs41} is equivalent to the relation $\tau_1 s_2 s_1 = s_2 s_1 \tau_2$, which is the base case for relations \eqref{eqs22}. Finally, observe that relation \eqref{eqs39} is the base case of relations \eqref{eqs21} in our original presentation for the monoid $VSTM_n$.

In the following lemmas, we show that each relation in the original presentation for $VSTM_n$ is satisfied, therefore proving Theorem~\ref{thm:reduced}. In each proof below, we underline the portion of the word that is modified in the subsequent next and specify the relations used. The first lemma deals with preparatory identities. 

 \begin{lemma} The following relations hold in $VSTM_{n}$, for all $i-j \geq 2$:
\begin{eqnarray} \nu_i \nu_{i-1}\ldots { \nu_{j+1}{\nu_j} \nu_{j+1}}\ldots \nu_{i-1} \nu_i =  \nu_j \nu_{j+1}\ldots \nu_{i-1}{\nu_i} \nu_{i-1}\ldots   \nu_{j+1} \nu_j. \label{eqs45} 
 \end{eqnarray}  
 \end{lemma} 

\begin{proof}
We prove the statement by induction on $i-j$, and rely heavily on relations~\eqref{eqs36} and~\eqref{eqs37}. If $i-j = 2$, the statement is equivalent to relation \eqref{eqs36}. Suppose that relation \eqref{eqs45} holds for $i-j = k$, where $k \geq 2$. Then, observe the following:
\begin{align*}
 \nu_i \nu_{i-1}\ldots \underline{ \nu_{j+1}{\nu_j} \nu_{j+1}}\ldots \nu_{i-1} \nu_i &\stackrel{\eqref{eqs36}}{=} \nu_i  \nu_{i-1}\ldots \nu_{j+2}\underline{\nu_j}\nu_{j+1} \underline{\nu_j} \nu_{j+2} \ \ldots \nu_{i-1}  \nu_i \\
&\stackrel{\eqref{eqs37}}{=} \nu_j ( \underline{\nu_i  \nu_{i-1}\ldots \nu_{j+2}\nu_{j+1} \nu_{j+2}  \ldots \nu_{i-1}  \nu_i}  )\nu_j\\
&\ =  \nu_j \nu_{j+1}\ldots \nu_{i-1}{\nu_i} \nu_{i-1}\ldots   \nu_{j+1} \nu_j. 
\end{align*}
In the last step above, we used the inductive hypothesis.
\end{proof}

\begin{lemma} \label{lemma:detour_rel}
For all $|i-j|=1$, the  relations  $\nu_is_j \nu_i=\nu_js_i \nu_j$  and  $\nu_i\tau_j \nu_i = \nu_j \tau_i \nu_j$  follow from the reduced presentation for $VSTM_{n}$.
\end{lemma}

\begin{proof}
It is clear that these relations hold, since they were used in the defining relations~\eqref{A14} and~\eqref{A15}. Nevertheless, we provide a proof for the second set of relations in the case
 $j = i+1$ and $i \geq 1$; the first set of relations follows by a similar argument. For $i = 1$, observe that the relation $\nu_1 \tau_2 \nu_1 = \nu_2 \tau_1 \nu_2$ follows from the defining relation $\tau_2 = \nu_1 \nu_2 \tau_1 \nu_2 \nu_1$ by multiplying both sides of it by $\nu_1$ and using that $\nu_1^2 = 1$. Now let $i \geq 2$, and note the following:
\begin{align*}
\nu_i \tau_{i+1} \nu_i \stackrel{\eqref{A15}}{=}& \underline{\nu_i  (\nu_i} \ldots \nu_2 \nu_1)(\nu_{i+1}\ldots \nu_3 \nu_2)\tau_1 (\nu_2 \nu_3 \ldots \nu_{i+1})(\nu_1\nu_2\ldots \underline{\nu_i) \nu_i}\\
\stackrel{\eqref{eqs38}}{=}&  (\nu_{i-1}\ldots \nu_2 \nu_1)(\underline{\nu_{i+1}}\ldots \nu_3 \nu_2) \tau_1 (\nu_2 \nu_3\ldots \underline{\nu_{i+1}})(\nu_1\nu_2\ldots \nu_{i-1}) \\
\stackrel{\eqref{eqs37}}{=}& \nu_{i+1}  (\underline{\nu_{i-1}\ldots \nu_2 \nu_1)(\nu_i\ldots \nu_3 \nu_2)\tau_1 (\nu_2 \nu_3 \ldots \nu_i)(\nu_1 \nu_2 \ldots \nu_{i-1}}) \nu_{i+1}\\
\stackrel{\eqref{A15}}{=}& \nu_{i+1} \tau_i \nu_{i+1}.
\end{align*}
This completes the proof.
\end{proof}

We observe that for $|i-j|=1$, the relations $\nu_is_j \nu_i=\nu_js_i \nu_j$ and $\nu_i\tau_j \nu_i = \nu_j \tau_i \nu_j$ are equivalent to the following relations:
\begin{equation} \label{eqs46}
s_i \nu_j \nu_i = \nu_j \nu_i s_j \text{ and respectively } \tau_i \nu_j \nu_i = \nu_j \nu_i \tau_j.  
\end{equation}

\begin{lemma}
For all  $1 \leq i \leq n-1$, the relations $\tau_i  s_i = s_i \tau_i$ and $s_i^2 = 1$ follow from the reduced presentation for $VSTM_{n}$.
\end{lemma}
 \begin{proof} Let $i >1$. We first use the defining relations~\eqref{A14} and~\eqref{A15} for $s_i$ and $\tau_i$ in terms of $s_1$ and $\tau_1$, respectively, followed by the virtual relations $\nu_i^2 = 1$, to obtain:
\begin{align*}
\tau_i s_i =& [ (\nu_{i-1}\ldots \nu_2 \nu_1)(\nu_{i}\ldots \nu_3 \nu_2)\tau_1  (\nu_2 \nu_3 \ldots \nu_{i}) \underline{(\nu_1\nu_2 \ldots \nu_{i-1}) ]\,  [ (\nu_{i-1}\ldots \nu_2\nu_1) }\cdot \\ 
& (\nu_{i}\ldots \nu_3 \nu_2)s_1 (\nu_2 \nu_3 \ldots \nu_{i})(\nu_1 \nu_2\ldots \nu_{i-1}) ]\\
\stackrel{\eqref{eqs38}}{=}& (\nu_{i-1}\ldots \nu_2 \nu_1)(\nu_{i}\ldots \nu_3 \nu_2)\tau_1  \underline{(\nu_2 \nu_3\ldots \nu_i)  (\nu_i \ldots \nu_3 \nu_2) } s_1   (\nu_2 \nu_3 \ldots \nu_i)(\nu_1\nu_2\ldots \nu_{i-1}) \\
\stackrel{\eqref{eqs38}}{=}&  (\nu_{i-1}\ldots \nu_2 \nu_1)(\nu_i \ldots \nu_3 \nu_2)\tau_1  s_1  (\nu_2 \nu_3 \ldots \nu_i)(\nu_1 \nu_2 \ldots \nu_{i-1}). 
\end{align*}
Using similar steps, we obtain the following
\begin{align*}
 s_i \tau_i=(\nu_{i-1}\ldots \nu_2 \nu_1)(\nu_i \ldots \nu_3 \nu_2) s_1  \tau_1(\nu_2 \nu_3 \ldots \nu_i)(\nu_1 \nu_2 \ldots \nu_{i-1}). 
\end{align*}
But since $\tau_1 s_1 = s_1 \tau_1$, we obtain that $\tau_i s_i =  s_i \tau_i$. Finally, it is clear that the relations $s_i^2 = 1$ hold for all $i$, due to relations \eqref{A14} and \eqref{eqs38}.
\end{proof}

\begin{lemma}\label{sigv-tauv}
For all $|i-j| \geq 2$, the following commuting relations follow from the reduced presentation for $VSTM_{n}$:
\begin{enumerate}
\item[(i)]  $s_i \nu_j  =  \nu_j s_i$ and $\tau_i \nu_j  =  \nu_j \tau_i$,
\item[(ii)] $s_i s_j = s_j s_i$,  $\tau_i  \tau_j = \tau_j \tau_i$ and  $s_i  \tau_j = \tau_j s_i$.
  \end{enumerate}
\end{lemma}

\begin{proof}
(i) By the defining relation~\eqref{A14} for $s_i$, we have:
  \[s_i \nu_j =(\nu_{i-1}\ldots
\nu_2 \nu_1)\,(\nu_i \ldots \nu_3 \nu_2)\,s_1\, (\nu_2 \nu_3 \ldots \nu_i)\,(\nu_1 \nu_2\ldots \nu_{i-1})\, \nu_j.\]  
Since $|i-j| \geq 2$, either $j\geq i+2$ or $j\leq i-2$. If $j\geq i+2$, then $\nu_j$  commutes  with all generators in the above
expression, thus $s_i \nu_j  =  \nu_j s_i$ in this case. 
 If $j\leq i-2$ we have:
\begin{align*}
s_i \nu_j  \stackrel{\eqref{A14}}{=}&  (\nu_{i-1}\ldots \nu_2\nu_1)(\nu_i\ldots \nu_3\nu_2)s_1 (\nu_2\nu_3\ldots \nu_i)(\nu_1\nu_2\ldots \nu_{i-1})\underline{\nu_j}  \\
\stackrel{(\ref{eqs37})}{=}&  (\nu_{i-1}\ldots \nu_1)(\nu_i\ldots \nu_2)s_1 (\nu_2\nu_3\ldots \nu_i)(\nu_1\nu_2 \ldots \nu_{j-1}\underline{\nu_j \nu_{j+1} \nu_j}  \nu_{j+2} \ldots \nu_{i-1})  \\ 
\stackrel{(\ref{eqs36})}{=}&  (\nu_{i-1}\ldots \nu_1)(\nu_i\ldots \nu_2)s_1 (\nu_2 \nu_3\ldots  \nu_i)(\nu_1\nu_2 \ldots \nu_{j-1} \underline{\nu_{j+1}} \nu_j  \nu_{j+1} \nu_{j+2} \ldots \nu_{i-1}) \\
\stackrel{(\ref{eqs37})}{=}&  (\nu_{i-1}\ldots \nu_1)(\nu_i\ldots \nu_3\nu_2)s_1 (\nu_2\nu_3\ldots  \nu_j \underline{ \nu_{j+1} \nu_{j+2} \nu_{j+1}}  \nu_{j+3} \ldots \nu_i) (\nu_1\nu_2 \ldots \nu_{i-1}) \\
\stackrel{(\ref{eqs36})}{=}&  (\nu_{i-1}\ldots \nu_1)(\nu_i\ldots \nu_3 \nu_2)s_1(\nu_2 \nu_3 \ldots  \nu_j\underline{ \nu_{j+2}} \nu_{j+1} \nu_{j+2} \nu_{j+3} \ldots \nu_i)  (\nu_1 \nu_2 \ldots \nu_{i-1}) \\
 \displaystyle  \mathop{=}_{\eqref{eqs40}}^{\eqref{eqs37}}&  (\nu_{i-1}\ldots \nu_1)(\nu_i\ldots \nu_{j+3} \underline{ \nu_{j+2}  \nu_{j+1} \nu_{j+2}} \nu_j\ldots \nu_2) s_1 (\nu_2\ldots  \nu_i) (\nu_1 \ldots \nu_{i-1}) \\
\stackrel{(\ref{eqs36})}{=}&  (\nu_{i-1}\ldots \nu_2 \nu_1)(\nu_i \ldots  \nu_{j+3} \underline{ \nu_{j+1}} \nu_{j+2} \nu_{j+1} \nu_j\ldots \nu_2)s_1 (\nu_2 \ldots  \nu_i)  (\nu_1 \ldots \nu_{i-1}) \\
\stackrel{(\ref{eqs37})}{=}&  (\nu_{i-1}\ldots \nu_{j+2} \underline{\nu_{j+1} \nu_j \nu_{j+1} } \nu_{j-1} \ldots \nu_2 \nu_1)  (\nu_i\ldots \nu_3 \nu_2)s_1 (\nu_2 \ldots  \nu_i) (\nu_1 \ldots \nu_{i-1}) \\
\stackrel{(\ref{eqs36})}{=}& (\nu_{i-1}\ldots \nu_{j+2} \underline{\nu_j }  \nu_{j+1}  \nu_j  \nu_{j-1} \ldots \nu_2 \nu_1)  (\nu_i\ldots \nu_3\nu_2) s_1(\nu_2\ldots  \nu_i) (\nu_1 \ldots \nu_{i-1}) \\
\stackrel{(\ref{eqs37})}{=}&  \nu_j (\nu_{i-1}\ldots  \nu_1)  (\nu_i\ldots \nu_2) s_1 (\nu_2\ldots  \nu_i) (\nu_1 \ldots \nu_{i-1}) \\
 \displaystyle  \stackrel{\eqref{A14}}{=} &  \nu_j s_i.  
\end{align*} 
Therefore, $s_i \nu_j = \nu_js_i$ for all $|i-j| \geq 2$. The relation $\tau_i \nu_j = \nu_j \tau_i$, where $|i-j| \geq 2$,  is verified in a similar way, by making use of the defining relation \eqref{A15} together with the relations \eqref{eqs36}, \eqref{eqs37} and \eqref{eqs40}.

(ii) We show first that $s_1 \tau_i = \tau_i s_1$ and $s_i \tau_1 = \tau_1 s_i$ for all $i\geq 3$. Note that we already have that $\tau_1 s_3 = s_3 \tau_1$ by relation~\eqref{eqs43}. That is, we have that:
\[ \tau_1 (\nu_2 \nu_3 \nu_1 \nu_2 s_1 \nu_2 \nu_1 \nu_3 \nu_2) = (\nu_2 \nu_3 \nu_1 \nu_2 s_1 \nu_2 \nu_1 \nu_3 \nu_2) \tau_1. \]
Multiplying both sides of the equality by $\nu_2 \nu_1 \nu_3 \nu_2$ on the left and by $\nu_2 \nu_3 \nu_1 \nu_2$ on the right, we obtain
\[  (\nu_2 \nu_1 \nu_3 \nu_2 \tau_1 \nu_2 \nu_3 \nu_1 \nu_2) s_1 =  s_1(\nu_2 \nu_1 \nu_3 \nu_2 \tau_1 \nu_2 \nu_3 \nu_1 \nu_2),  \]
which is equivalent to $\tau_3 s_1 = s_1 \tau_3$, according to identity~\eqref{A15}.

Now let $i \geq 4$, and observe the following:
\begin{align*}
s_1 \tau_i & \stackrel{\eqref{A15}}{=}  \underline{s_1}(\nu_{i-1}\ldots \nu_2\nu_1)(\nu_i\ldots \nu_3 \nu_2)\tau_1 (\nu_2 \nu_3\ldots \nu_i)(\nu_1\nu_2\ldots \nu_{i-1}) \\
& \stackrel{\eqref{eqs40}}{=} (\nu_{i-1}\ldots  \underline{s_1 \nu_2\nu_1})(\nu_i\ldots \nu_3 \nu_2)\tau_1 (\nu_2 \nu_3\ldots \nu_i)(\nu_1\nu_2\ldots \nu_{i-1}) \\
&\stackrel{\eqref{eqs46}}{=}  (\nu_{i-1}\ldots  \nu_2\nu_1 \underline{s_2})(\nu_i\ldots \nu_3 \nu_2)\tau_1 (\nu_2 \nu_3\ldots \nu_i)(\nu_1\nu_2\ldots \nu_{i-1}) \\
& = (\nu_{i-1}\ldots  \nu_2\nu_1 )(\nu_i\ldots \underline{s_2\nu_3 \nu_2})\tau_1 (\nu_2 \nu_3\ldots \nu_i)(\nu_1\nu_2\ldots \nu_{i-1})\\
& \stackrel{\eqref{eqs46}}{=} (\nu_{i-1}\ldots  \nu_2\nu_1 )(\nu_i\ldots \nu_3 \nu_2 s_3)\tau_1 (\nu_2 \nu_3\ldots \nu_i)(\nu_1\nu_2\ldots \nu_{i-1})\\
& = (\nu_{i-1}\ldots  \nu_2\nu_1 )(\nu_i\ldots \nu_3 \nu_2) s_3\tau_1 (\nu_2 \nu_3\ldots \nu_i)(\nu_1\nu_2\ldots \nu_{i-1}).
\end{align*}
In the fourth step above we use part (i) of this lemma, to commute $s_2$ with the product $\nu_i \nu_{i-1} \ldots \nu_4$. Using similar steps, it is now easy to see that the following holds:
\begin{align*}
\tau_i s_1 &= (\nu_{i-1}\ldots \nu_2\nu_1)(\nu_i\ldots \nu_3 \nu_2)\tau_1 (\nu_2 \nu_3\ldots \nu_i)(\nu_1\nu_2\ldots \nu_{i-1}) s_1\\
 & = (\nu_{i-1}\ldots  \nu_2\nu_1 )(\nu_i\ldots \nu_3 \nu_2) \tau_1 s_3 (\nu_2 \nu_3\ldots \nu_i)(\nu_1\nu_2\ldots \nu_{i-1}).
\end{align*}
Comparing the equalities we just derived and using that $s_3 \tau_1 = \tau_1 s_3$, we conclude that $s_1 \tau_i = \tau_i s_1$ for all $i \geq 3$. A similar proof can be used to show that $s_i \tau_1 = \tau_1 s_i$, $s_i s_1 = s_1 s_i$ and $\tau_i \tau_1 = \tau_1 \tau_i$, for all $i\geq 3$.

We will now prove that $s_i \tau_j = \tau_j s_i$, for all $|i-j| \geq 2$. Let $|i-j| \geq 2$. 

Case 1. If $i-j \geq 2$, we observe that:
\[ s_i \tau_j \stackrel{\eqref{A15}}{=} s_i(\nu_{j-1}\ldots \nu_2\nu_1)(\nu_j\ldots \nu_3 \nu_2)\tau_1 (\nu_2 \nu_3\ldots \nu_j)(\nu_1\nu_2\ldots \nu_{j-1}), \]
and note that $s_i$ commutes with each $\nu_1, \nu_2, \dots, \nu_j$, by part (i) in this lemma. Moreover, $s_i$ commutes with $\tau_1$, as proved above. Hence, $s_i \tau_j =\tau_js_i$, for $i-j \geq 2$.

 Case 2. If $j - i \geq 2$, then we write:
\[s_i\tau_j \stackrel{\eqref{A14}}{=} (\nu_{i-1}\ldots \nu_2\nu_1)(\nu_i\ldots \nu_3 \nu_2)\tau_1 (\nu_2 \nu_3\ldots \nu_i)(\nu_1\nu_2\ldots \nu_{i-1}) \tau_j.\]
Since $j \geq i+2$, we know that $\tau_j$ commutes with each of the virtual generators $\nu_1, \nu_2, \dots, \nu_j$, again due to part (i) of Lemma~\ref{sigv-tauv}. The generator $\tau_j$ also commutes with $\tau_1$, as explained above. Therefore, $s_i \tau_j = \tau_j s_i$, for $j-i \geq 2$ as well.

The proofs that relations $s_i s_j = s_j s_i$ and  $\tau_i  \tau_j = \tau_j \tau_i$ hold for all $|i-j|\geq 2$ follow similarly, and thus they are omitted to avoid repetition.
\end{proof}

\begin{lemma}
The relations $s_j  s_i \tau_j = \tau_{i} s_j s_{i}$, where $|i-j|=1$, follow from the reduced presentation for $VSTM_{n}$.
\end{lemma}

\begin{proof}
The base case for this set of relations is represented by the relation ~\eqref{eqs41}. We show that the relations hold for the case $j = i+1$ and $i \geq $2.  The proof for the case $j = i-1$ and $i\geq 3$ follows similarly. For the right hand side of the identity, we have:
\begin{align*}
\tau_i  s_{i+1} s_i  \displaystyle  \mathop{=}_{\eqref{A15}}^{\eqref{A14}}&  [ (\nu_{i-1}\ldots \nu_2\nu_1)(\nu_{i}\ldots \nu_3\nu_2)\tau_1  (\nu_2\nu_3\ldots \nu_{i}) \underline{(\nu_1\nu_2\ldots \nu_{i-1}) }] \cdot\\
& [\underline{ (\nu_i\ldots \nu_2 \nu_1) }(\nu_{i+1}\ldots \nu_3 \nu_2) s_1 (\nu_2 \nu_3 \ldots \nu_{i+1}) \underline{ (\nu_1 \nu_2 \ldots \nu_i) }]\cdot\\
&  [ \underline{(\nu_{i-1} \ldots \nu_2 \nu_1) } (\nu_{i} \ldots \nu_3 \nu_2) s_1  (\nu_2 \nu_3 \ldots \nu_{i})(\nu_1 \nu_2 \ldots \nu_{i-1}) ]\\
\stackrel{\eqref{eqs45}}{=}&  (\nu_{i-1}\ldots \nu_2 \nu_1)(\nu_{i}\ldots \nu_3 \nu_2) \tau_1  (\underline{\nu_2 \nu_3 \ldots \nu_{i}) (\nu_i \ldots \nu_2 } \nu_1 \nu_2 \ldots \nu_i) (\nu_{i+1}\ldots \nu_3 \nu_2)s_1 \cdot \\
&(\nu_2 \nu_3 \ldots \nu_{i+1})  (\nu_i \ldots  \nu_2 \nu_1 \underline{ \nu_2 \ldots \nu_i) (\nu_{i}\ldots \nu_3\nu_2)} s_1  (\nu_2\nu_3\ldots \nu_{i})(\nu_1\nu_2\ldots \nu_{i-1}) \\
\stackrel{\eqref{eqs38}}{=}& (\nu_{i-1}\ldots \nu_2\nu_1)(\nu_{i}\ldots \nu_3\nu_2)(\tau_1   \nu_1 \underline{ \nu_2\ldots \nu_i) (\nu_{i+1}\ldots \nu_3\nu_2) } s_1  \cdot \\ 
&  \underline{ (\nu_2 \nu_3\ldots \nu_{i+1}) (\nu_i \ldots  \nu_2 } \nu_1  s_1 )(\nu_2\nu_3\ldots \nu_{i})(\nu_1\nu_2\ldots \nu_{i-1}) \\
\stackrel{\eqref{eqs45}}{=}&  (\nu_{i-1}\ldots \nu_2\nu_1)(\nu_{i}\ldots \nu_3\nu_2)(\underline {\tau_1   \nu_1}  \nu_{i+1} \ldots \nu_3 \nu_2 \nu_3\ldots \nu_{i+1}) \underline{ s_1}  \cdot \\ 
&(\nu_{i+1} \ldots \nu_3 \nu_2 \nu_3  \ldots  \nu_{i+1}  \underline{ \nu_1  s_1 } )  (\nu_2 \nu_3\ldots \nu_{i})(\nu_1 \nu_2\ldots \nu_{i-1}) \\
\displaystyle  \mathop{=}_{\eqref{eqs37}}^{\eqref{eqs40}}  &   (\nu_{i-1}\ldots v_2 \nu_1)(\nu_{i}\ldots \nu_3 \nu_2) (\nu_{i+1} \ldots \nu_3) (\tau_1   \nu_1 \nu_2  \underline{\nu_3 \ldots \nu_{i+1})}\cdot \\
  & \underline{(\nu_{i+1} \ldots \nu_3 } s_1 \nu_2 \nu_1  s_1  \nu_3 \ldots  \nu_{i+1}) (\nu_2 \nu_3 \ldots \nu_{i})(\nu_1 \nu_2\ldots \nu_{i-1})  \\
\stackrel{\eqref{eqs38}}{=}&  (\nu_{i-1}\ldots \nu_1)(\nu_{i}\ldots \nu_2) ( \nu_{i+1} \ldots \nu_3) ( \tau_1   \nu_1 \nu_2  s_1 \nu_2 \nu_1  s_1) \cdot \\
&( \nu_3  \ldots  \nu_{i+1} )     (\nu_2\ldots \nu_{i})(\nu_1\ldots \nu_{i-1}).
\end{align*}

Now we will consider the left hand side of the identity, and we start by replacing $s_i$ using the defining relation~\eqref{A14} and inserting the identity element, which is then replaced by $\nu_1 \nu_1$ and $(\nu_{i+1}\ldots \nu_3) (\nu_3 \ldots  \nu_{i+1})$, as shown below:
\begin{align*}
s_{i+1} \underline{ s_i} \tau_{i+1} \stackrel{\eqref{A14}}{=}& s_{i+1} (\nu_{i-1}\ldots \nu_1)  (\nu_i \ldots  \nu_2 \underline{(1)}) s_1 \underline{(1)(1)}(\nu_2 \ldots \nu_{i}) (\nu_1\ldots \nu_{i-1})  \tau_{i+1} \\
= \,\,\,& \underline{ s_{i+1} } (\nu_{i-1}\ldots \nu_1)  (\nu_i \ldots \nu_2 \nu_1 \nu_1) s_1 \underline{(\nu_{i+1}\ldots \nu_3)}   (\nu_3 \ldots  \nu_{i+1} )  (\underline{\nu_1} \nu_1 )\cdot \\
&( \nu_2\ldots \nu_{i}) (\nu_1\ldots \nu_{i-1}) \underline{\tau_{i+1}}. 
\end{align*}
Employing the commuting relations in Lemma~\ref{sigv-tauv} and those in Equations~\eqref{eqs37} and~\eqref{eqs40}, we obtain:
\begin{align*}
s_{i+1}  s_i \tau_{i+1} = \,\,\,&  (\nu_{i-1}\ldots \nu_1) s_{i+1}  (\nu_i \ldots \nu_{1})(\nu_{i+1}\ldots \nu_3)\underline{(1)}  \nu_1  s_1 \,\nu_1 \underline{(1)} (\nu_3 \ldots  \nu_{i+1} ) \cdot \\
& (\nu_1\ldots \nu_{i}) \tau_{i+1} (\nu_1\ldots \nu_{i-1}) \\
= \,\,\, &  (\nu_{i-1}\ldots \nu_1) \underline{s_{i+1}} (\nu_i \ldots \nu_{1})(\nu_{i+1}\ldots \nu_3) ( \nu_2   \nu_2 ) \nu_1  s_1 \nu_1 (\nu_2   \nu_2 )( \nu_3 \ldots  \nu_{i+1} )  \cdot \\
&(\nu_1\ldots \nu_{i}) \underline{\tau_{i+1}}(\nu_1\ldots \nu_{i-1})\\
\displaystyle  \mathop{=}_{\eqref{A15}}^{\eqref{A14}}  &  (\nu_{i-1}\ldots \nu_1)  (\nu_{i}\ldots \nu_1) ( \nu_{i+1} \ldots \nu_2 ) s_1 \underline{(\nu_2 \ldots \nu_{i+1})(\nu_1 \ldots \nu_{i})} \cdot \\
&\underline{(\nu_i \ldots \nu_{1})(\nu_{i+1}\ldots \nu_{2}) } \nu_2 \nu_1  s_1 \nu_1 \nu_2  \underline{( \nu_2 \ldots  \nu_{i+1} ) }\cdot \\
& \underline{(\nu_1\ldots \nu_{i}) (\nu_{i}\ldots \nu_1) ( \nu_{i+1} \ldots \nu_2 )}\tau_1 ( \nu_2 \ldots  \nu_{i+1} ) (\nu_1\ldots \nu_{i}) (\nu_1\ldots \nu_{i-1}).
\end{align*}

We next use relations~\eqref{eqs38} to replace the product $( \nu_2 \ldots  \nu_{i+1} )(\nu_1\ldots \nu_i) (\nu_i\ldots \nu_1) ( \nu_{i+1} \ldots \nu_2 )$ with identity. We then use the commuting relations \eqref{eqs37} to move two instances of $\nu_1$ to desired locations. In the last step below, we employ the relation~\eqref{eqs41}.
\begin{align*}
s_{i+1}  s_i \tau_{i+1} 
 \stackrel{\eqref{eqs38}}{=}& (\nu_{i-1}\ldots \nu_1)(\nu_{i}\ldots \underline{\nu_1}) ( \nu_{i+1} \ldots \nu_2 ) s_1 \nu_2 \nu_1  s_1 \nu_1 \nu_2 \tau_1 ( \nu_2 \nu_3 \ldots  \nu_{i+1} ) \cdot \\
 & (\underline{\nu_1}\ldots \nu_{i})(\nu_1\ldots \nu_{i-1}) \\%
 \stackrel{\eqref{eqs37}}{=}&  (\nu_{i-1}\ldots \nu_1)(\nu_{i}\ldots \nu_2) ( \nu_{i+1} \ldots \nu_3)(\underline{\nu_1 \nu_2 s_1  \nu_2 \nu_1 ) s_1 (\nu_1 \nu_2 \tau_1 \nu_2  \nu_1 }  ) \cdot
  \\
& ( \nu_3 \ldots  \nu_{i+1} )   (\nu_2\ldots \nu_{i})(\nu_1\ldots \nu_{i-1}) \\
\stackrel{\eqref{eqs41}}{=}&  (\nu_{i-1}\ldots \nu_1)(\nu_{i}\ldots \nu_2) ( \nu_{i+1} \ldots \nu_3) ( \tau_1   \nu_1 \nu_2  s_1 \nu_2 \nu_1  s_1) \cdot \\
&( \nu_3 \ldots  \nu_{i+1} )   (\nu_2\ldots \nu_{i})(\nu_1\ldots \nu_{i-1}).
\end{align*}
Therefore, $s_{i+1} s_i \tau_{i+1}  = \tau_i  s_{i+1} s_i$ for all $i \geq $2. 
\end{proof}

\subsection{Yet Another Presentation for $VSTM_n$}
In this section, we introduce a new presentation for the $n$-stranded virtual singular twin monoid, $VSTM_n$. This presentation uses as generators a special class of virtual singular twin braids, which we  refer to as connecting strings. What distinguishes these elements of $VSTM_n$ is that their associated permutation in $S_n$ is the identity permutation; hence they are virtual singular pure twin braids in the sense of Definition~\ref{Virtual Singular Pure Twin Group}. We note that similar approaches were used  by L. H. Kauffman and S. Lampropoulou in~\cite{KL2} for the virtual braid group, and by C. Caprau and S. Zepeda in~\cite{Caprau1} for the virtual singular braid monoid.

  \begin{definition}\label{def:cstrings}
  The \textit{connecting strings} $\mu_i$ and $\gamma_i$, where $1 \leq i \leq n-1$, are $n$-stranded virtual singular twin braids defined as follows, and depicted in Fig.~\ref{fig:connecting-strings}.
   \begin{eqnarray} \label{eqs_connecting_strings}
   \mu_i: = s_i \nu_i ,  \hspace{0.5cm} \gamma_i: = \tau_i \nu_i, \,\,\, \text{where} \,\,\,1 \leq i \leq n-1.
   \end{eqnarray}
    \end{definition} 
    
     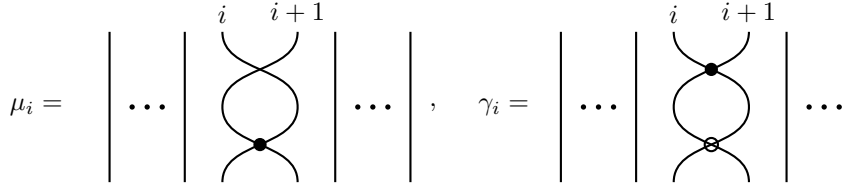
\begin{figure}[h!]   
\begin{tikzpicture}
	\draw[thick] (5.5,-1)--(5.5,1); 
   \fill (5,0) circle(1pt) (4.8,0) circle(1pt)(5.2,0)circle(1pt);       
    \draw[thick] (4.5,-1)--(4.5,1);
    \draw[thick] (8.5,-1)--(8.5,1); 
    \fill (8,0) circle(1pt) (8.2,0) circle(1pt)(7.8,0)circle(1pt);       
   \draw[thick] (7.5,-1)--(7.5,1);    
	\draw[thick] (7,-1) to[out=90, in=-90] (6,0);
    \draw[thick] (6,-1) to[out=90, in=-90] (7,0);
    \draw[thick] (7,0) to[out=90, in=-90] (6,1);
    \draw[thick] (6,0) to[out=90, in=-90] (7,1);
\draw[thick] (6.5,-0.5) circle (2.5pt); 
   \fill[black] (6.5,0.5) circle (2.5pt); 
    \node[above] at(6,1){$i$};
    \node[above] at (7,1) {$i+1$};
     \node[left] at(-2,0){$\mu_i=$};
	\node at (2.8,0){,};    
 	\draw[thick] (-1.5,-1)--(-1.5,1); 
    \fill (-1,0) circle(1pt) (-1.2,0) circle(1pt)(-0.8,0)circle(1pt);       
    \draw[thick] (-.5,-1)--(-.5,1);
    \draw[thick] (2.5,-1)--(2.5,1); 
    \fill (2,0) circle(1pt) (2.2,0) circle(1pt)(1.8,0)circle(1pt);       
    \draw[thick] (1.5,-1)--(1.5,1);
    \draw[thick] (0,0) to[out=90, in=-90] (1,1);
    \draw[thick] (1,0) to[out=90, in=-90] (0,1);
     \draw[thick] (1, -1) to[out=90, in=-90] (0,0);
    \draw[thick] (0,-1) to[out=90, in=-90] (1,0);
   \fill[black] (0.5,-0.5) circle (2.5pt); 
       \node[above] at(0,1){$i$};
    \node[above] at(1,1){$i+1$};
     \node[left] at(4.2,0){$\gamma_i=$};
  \end{tikzpicture}
   \caption{The connecting strings $\mu_i$ and $\gamma_i$}
   \label{fig:connecting-strings}
 \end{figure}
 
It is easy to observe that the elements $s_i$ and $\tau_i$ can be described in terms of the connecting strings as $s_i=\mu_{i} \nu_i$ and $\tau_i=\gamma_{i} \nu_{i}.$
Thus, the connecting strings $\mu_{i}$ and $\gamma_{i}$, together with the virtual generators $\nu_i$, for all $1\leq i \leq n-1$, can be used as an alternative generating set for $VSTM_n$ and $VST_n$. The following lemma provides a collection of relations satisfied by the connecting strings. In the proof of each part of the lemma, we underline again the portion of the word that is modified in the subsequent step, using a defining relation from the original presentation for $VSTM_n$, according to Definition~\ref{def:VSTM}. Note that the set of relations (i) in the following lemma introduces a version of the detour relations for the connecting strings, while the last set of relations (iv) are the commuting relations for the elements $\mu_i,\gamma_i$ and $\nu_i$.

\begin{lemma}\label{lemma:rel_connecting_strings}
The following relations hold in $VSTM_n$.
\begin{enumerate}
\item [(i)]  $\nu_i\mu_{j} \nu_{i} = \nu_{j} \mu_{i} \nu_{j}$ and $ \nu_i\gamma_{j} \nu_{i} = \nu_{j}\gamma_{i} \nu_{j}$, where $|i-j|=1$
\item [(ii)] $\mu_{j}( \nu_{j}\mu_{i} \nu_{j})\gamma_i=\gamma_{i}( \nu_{j}\mu_{i} \nu_{j})\mu_{j}$, where $|i-j|=1$
 \item [(iii)] $\mu_{i} \nu_i\gamma_i = \gamma_i \nu_i\mu_{i}$, for all $1\leq i \leq n-1$
 \item [(iv)] $\alpha_i\beta_j=\beta_j\alpha_i,\,\text{for } \alpha_i,\beta_i\in \{\mu_i,\gamma_i, \nu_i\}$, where $|i-j| \geq 2$.

 \end{enumerate}
\end{lemma}

\begin{proof} 

(i) The desired relations follow from the detour relations of $VSTM_n$, as demonstrated below:
\[
\nu_i\underline{\mu_{j}} \nu_i = \nu_i s_{j}\nu_{j}\nu_i\underline{1} = \nu_is_{j}\underline{\nu_{j}\nu_i \nu_{j}}\nu_{j} = \underline{\nu_is_{j}\nu_{i}}\nu_{j}\nu_{i}\nu_{j} = \nu_{j}\underline{s_{i} (\nu_{j}\nu_{j})\nu_{i}}\nu_{j}  = \nu_{j}\mu_i \nu_{j},
\]
\[
\nu_i\underline{\gamma_{j}} \nu_i = \nu_i\tau_{j}\nu_{j}\nu_i \underline{1}= \nu_i\tau_{j}\underline{\nu_{j}\nu_i \nu_{j}}\nu_{j} = \underline{\nu_i\tau_{j}\nu_{i}}\nu_{j}\nu_{i}\nu_{j} = \nu_{j}\underline{\tau_{i}(\nu_{j}\nu_{j})\nu_{i}}\nu_{j} = \nu_{j}\gamma_i \nu_{j}.
\]

(ii)  We remark that the relations in part (i) are equivalent to the following relations:
\[ \mu_{j} \nu_{i} \nu_{j} = \nu_i \nu_{j} \mu_{i} \,\,  \text{and} \,\, \gamma_{j} \nu_{i} \nu_{j} = \nu_{i} \nu_{j}\gamma_{i}.\]
This will be used in the proof of the relations in (ii). In the first and last steps below, we employ the defining relations for the connecting strings $\mu_i$ and $\gamma_i$, according to Definition~\ref{def:cstrings}.
\begin{align*}
\mu_{j}( \nu_{j}\mu_{i}  \nu_{j})\gamma_i &=s_{j}\underline{ \nu_{j} \nu_{j}}s_i \underline{ \nu_i  \nu_{j}\tau_i}  \nu_i
=\underline{s_{j}s_i\tau_{j}}\,\,\underline{ \nu_i  \nu_{j}  \nu_i}
=\tau_{i}s_{j}\underline{s_{i} \nu_{j} \nu_{i}} \nu_{j}\\
&=\tau_{i}s_{j}  \nu_{j}  \nu_{i}\underline{1}s_{j}  \nu_{j}
=\ \tau_{i} s_{j}\underline{ \nu_{j} \nu_{i} \nu_{j}} \nu_{j}s_{j} \nu_{j}
= \tau_{i}\underline{s_{j} \nu_{i}  \nu_{j}}  \nu_{i}  \nu_{j}s_{j}  \nu_{j}\\
&= ( \underline{\tau_{i}  \nu_{i}} )  \nu_{j} (\underline{s_{i}  \nu_{i}} ) \nu_{j} (\underline{s_{j} \nu_{j}})
= \gamma_i( \nu_{j}\mu_i  \nu_{j})\mu_{j}.
\end{align*}

(iii) Let $1\leq i \leq n-1$. Using the defining relations for $\mu_i$ and $\gamma_i$, the identity relation $\nu_i^2=1$, together with relation $s_i\tau_i=\tau_is_i$, which hold in $VSTM_n$, we obtain the desired result:
\[
\mu_i \nu_i\gamma_i = (s_i\underline{\nu_i) \nu_i}(\tau_i \nu_i)
=\underline{s_i\tau_i} \nu_i
= \tau_i\underline{1}s_i \nu_i
= (\tau_i \nu_i) \nu_i(s_i \nu_i)
= \gamma_i \nu_i\mu_i.
\]

(iv) The commuting relations for the connecting strings, as stated in  (iv), follow immediately from the commuting relations given in the original presentation for $VSTM_n$.
\end{proof}

We are now ready to define a monoid $M_n$ generated by the connecting strings, together with the virtual elements $\nu_i$. Note that the elements $\nu_i$ and $\mu_i$ are invertible, with $\nu_i^{-1} = \nu_i$ and $\mu_i^{-1} = \nu_is_i$.
 The relations established in the preceding lemma become defining relations for this monoid, along with the defining relations of $VSTM_n$ involving only the virtual generators. We will then show that $M_n$ is isomorphic to $VSTM_n$, which in turn yields an alternative presentation for $VSTM_n$.

\begin{definition} \label{def:mn}
Let $M_n$ be the monoid with the following presentation using generators $\{\mu_i, \mu_i^{-1},  \gamma_i, \nu_i  \, \big{\vert}  \,1\leq i \leq {n-1}\}$ and relations:
     \begin{align}
     \nu_i^2=1 \,\, &\text{and}\,\, \, \mu_i\mu_i^{-1}=1=\mu_i^{-1}\mu_i, \,\,\, \text{for all } i \label{eq:mnid}\\
       \nu_i \nu_j \nu_i&= \nu_j \nu_i \nu_j,  \,\,  |i-j|=1 \label{eq:mnv3}\\
       \nu_i \mu_j \nu_i&= \nu_j \mu_i \nu_j,  \,\,  |i-j|=1\label{eq:mnvr3}\\
 \nu_i \gamma_j \nu_i&= \nu_j \gamma_i \nu_j, \,\,  |i-j|=1\label{eq:mnvs3}\\
       \mu_j(\nu_j \mu_i \nu_j) \gamma_i &=\gamma_i (\nu_j \mu_i \nu_j) \mu_j, \,\,   |i-j|=1\label{eq:mnrs31}\\
         \mu_i \nu_i\gamma_i &=\gamma_i \nu_i \mu_i,  \,\,\, \text{for all } i  \label{eq:mnr1}\\
    \alpha_i\beta_j&= \beta_j\alpha_i,\,\,\,\,|i-j| \geq 2,\,\,  \text{ for all } \alpha_i, \beta_i \in \{\mu_i, \gamma_i, \nu_i\}\label{eq:mnfc}
\end{align} 
\end{definition}

\begin{theorem}\label{theorem:iso}
The monoids $M_n$ and $VSTM_n$ are isomorphic.
\end{theorem}   

\begin{proof}
We define the map $F: M_n \longrightarrow VSTM_n$ given by 
\[F(\nu_i) = \nu_i, \, F(\mu_i) = s_i \nu_i  \,\, \text{and} \,  F(\gamma_{i})=\tau_i \nu_i,\]
and extended to all elements of $M_n$ by imposing that is a homomorphism. By Lemma~\ref{lemma:rel_connecting_strings},  together with the definition for the monoid $M_n$, the homomorphism $F$ preserves the relations for $M_n$. To show that $F$ is an isomorphism, consider the map $G: VSTM_n\longrightarrow M_n$ defined on the generators of $VSTM_n$ by, 
 \[G(\nu_i) = \nu_i,\, G(s_i) = \mu_i \nu_i \, \, \text{and}\, G(\tau_i) = \gamma_i \nu_i,\]
and extended to all elements of $VSTM_n$ by imposing that it is a homomorphism.
It is straightforward to verify that $G$ preserves the defining relations of $VSTM_n$. Moreover, we observe that $F\circ G$ and $G\circ F$ are the identity maps on $VSTM_n$ and $M_n$, respectively. Therefore, the monoid $M_n$ is isomorphic to the virtual singular twin monoid $VSTM_n$.
\end{proof}

We now use the detour move to obtain a reduced presentation for the monoid $M_n$ by expressing the generators $\mu_i$ and $\gamma_i$ in terms of 
 $\mu_1$, and $\gamma_1$, respectively, as follows:
\[\mu_{i+1}= (\nu_i\dots \nu_1)(\nu_{i+1}\dots \nu_3 \nu_2)\mu_1(\nu_2 \nu_3 \dots \nu_{i+1})(\nu_1 \nu_2 \dots  \nu_i)\]
\[\gamma_{i+1} = (\nu_i\dots \nu_1)(\nu_{i+1}\dots \nu_3 \nu_2)\gamma_1(\nu_2 \nu_3 \dots \nu_{i+1}) (\nu_1 \nu_2\dots \nu_i).\]

This allows us to rewrite all relations in Definition~\ref{def:mn} using only the generators $\mu_1$ and $\gamma_1$ together with the virtual generators $\nu_i$ for $1\leq i \leq n-1$. Hence, the monoid $M_n$ admits a presentation with fewer generators. Consequently, the virtual singular twin monoid $VSTM_n$ has a reduced presentation in terms of the connecting strings $\mu_1$ and $\gamma_1$, along with the virtual generators $\nu_i$, as stated below.

\begin{theorem}\label{prop:redvssn}
The virtual singular twin monoid $VSTM_n$ admits the following reduced presentation with $\{\nu_1, \nu_2,\dots, \nu_{n-1}\}$ and connecting strings $\{\mu_1, \gamma_1\}$ as generators, and subject to the following relations:
\begin{align}
     \nu_i^2=1 \,\, &\text{and}\,\, \,  \mu_{1}^{-1}\mu_{1}=1 = \mu^{-1}_{1}\mu_{1}\\
          \nu_i \nu_j \nu_i& = \nu_j \nu_i \nu_j, \,\, |i-j|=1\\
      (\nu_1\nu_2\mu_{1}\nu_2\nu_1)(\nu_2\mu_{1}\nu_2)\gamma_{1}&=\gamma_{1}(\nu_2\mu_{1}\nu_2)(\nu_1\nu_2\mu_{1}\nu_2\nu_1) \\
           \mu_{1}\nu_1\gamma_{1}&=\gamma_{1}\nu_1\mu_{1}\\
       \nu_i\nu_j&=\nu_j\nu_i, \,\, |i-j| \geq 2 \label{eq:fcomv}\\
   \mu_{1} \nu_i  = \nu_i \mu_1\,\,\, &\text{and}\,\, \,\gamma_1\nu_i=\nu_i\gamma_1, \,\, i \geq 3\\
        \gamma_{1} (\nu_2 \nu_1\nu_3 \nu_2 \gamma_{1} \nu_2 \nu_3 \nu_1 \nu_2) &=(\nu_2 \nu_1\nu_3 \nu_2 \gamma_{1} \nu_2 \nu_3 \nu_1 \nu_2)\gamma_{1}   \\
        \gamma_{1} (\nu_2 \nu_1\nu_3 \nu_2 \mu_{1} \nu_2 \nu_3 \nu_1 \nu_2)& =(\nu_2 \nu_1\nu_3 \nu_2 \mu_{1} \nu_2 \nu_3 \nu_1 \nu_2)\gamma_{1}  \\
        \mu_{1} (\nu_2 \nu_1\nu_3 \nu_2 \mu_{1} \nu_2 \nu_3 \nu_1 \nu_2) &=(\nu_2 \nu_1\nu_3 \nu_2 \mu_{1} \nu_2 \nu_3 \nu_1 \nu_2)\mu_{1}   \label{eq:fcom_mu}
 \end{align} 
\end{theorem}

\begin{proof}
The proof is similar to that of Theorem~\ref{thm:reduced}. We remark that we do not require the generators $\mu_i$ and $\gamma_i$ for $2 \leq i \leq n-1$, since we have defined them in terms $\mu_1, \gamma_1$ and the virtual generators. With respect to the relations, we observe that they coincide with the relations in Definition~\ref{def:mn}, after fixing $i=1$ whenever the relation involves the generators $\mu_i$ or $\gamma_i$. When a relation involves both indices $i$ and $j$, we assume that it occurs at the left most portion of the braid and therefore use the smallest possible subscript for $j$. In particular, for the relations requiring $|i-j|=1$, we fix $j=2$ and express the corresponding generator in terms of a generator with subscript 1. For example, in the identity \eqref{eq:mnvs3}, we use $\nu_1\nu_2\mu_{1}\nu_2\nu_1$ in place of $\mu_2$. Similarly, for relations requiring $|i-j| \geq 2$, we fix $j=3$ and again express the corresponding generator in terms of a generator with subscript 1. Any relation occurring elsewhere in the braid follows from these relations by the detour move. 

We do not include the identities~\eqref{eq:mnvr3} and~\eqref{eq:mnvs3} in this reduced presentation, since they were used to provide the defining relations for $\mu_{i+1}$ and $\gamma_{i+1}$. The remaining relations of Definition~\ref{def:mn} are included with subscripts adjusted, as described above. The commuting relations~\eqref{eq:mnfc} are now represented by the relations~\eqref{eq:fcomv}--\eqref{eq:fcom_mu}. Therefore, the statement holds.
\end{proof}

\textbf{Acknowledgments.} C. Caprau was partially suported by NSF-RUI grant DMS 2204386.

\end{document}